\documentclass[10pt,xcolor=pdftex,a4paper,oneside,centertags,reqno,endnotes]{amsart}

\usepackage{apptools}
\usepackage{chngcntr}

\usepackage{esint}
\usepackage{endnotes}
\usepackage{latexsym}
\usepackage{amsmath}
\usepackage{amsfonts}
\usepackage{amssymb}   
\usepackage{epsfig}
\usepackage{amscd}
\usepackage{amsthm}
\usepackage{enumerate}
\usepackage{eucal}
\usepackage[utf8x]{inputenc}
\usepackage{graphics}
\usepackage[colorlinks = true]{hyperref}
\usepackage{amsrefs}

\usepackage{pdfsync}
\usepackage{color}
\usepackage{mathrsfs}

\usepackage{stackengine}

\newcommand{\dashint}{\,\ThisStyle{\ensurestackMath{%
\stackinset{c}{.2\LMpt}{c}{.5\LMpt}{\SavedStyle-}{\SavedStyle\phantom{\int}}}%
\setbox0=\hbox{$\SavedStyle\int\,$}\kern-\wd0}\int}

\usepackage[T1]{fontenc}
\usepackage{mathrsfs}
\usepackage{epsfig}
\usepackage{xypic}
\usepackage{verbatim}
\usepackage{multicol}
\usepackage{wrapfig}
\usepackage{pgf,tikz}
\usepackage{mathrsfs}
\usetikzlibrary{arrows}

\numberwithin{equation}{section}

\newtheorem{theorem}{Theorem}[section]
\newtheorem{proposition}[theorem]{Proposition}
\newtheorem{lemma}[theorem]{Lemma}
\newtheorem{corollary}[theorem]{Corollary}
\theoremstyle{definition}
\newtheorem{definition}[theorem]{Definition}
\newtheorem{remark}[theorem]{Remark}

\renewcommand{\Re}{\mathrm{Re}}

\newcommand{\Dom}{\mathrm{Dom}}

\newcommand{\N}{\mathbb{N}}
\newcommand{\Z}{\mathbb{Z}}
\newcommand{\R}{\mathbb{R}}
\newcommand{\C}{\mathbb{C}}

\usepackage{textcomp}
\usepackage{dsfont}
\usepackage{latexsym}
\usepackage{amssymb}
\usepackage{amsthm}
\usepackage{amsmath}
\DeclareMathAlphabet{\mathpzc}{OT1}{pzc}{m}{en}
\usepackage{yfonts}
\usepackage{xfrac}
\usepackage{fge}

\usepackage{fancyhdr}

\usepackage{newlfont}
\usepackage{graphicx}

\usepackage{mathtools}
\usepackage{comment}
\usepackage{indentfirst}
\usepackage{braket}
\usepackage{scalerel}

\DeclareMathOperator{\supp}{supp}

\DeclarePairedDelimiter{\abs}{\lvert}{\rvert}

\DeclarePairedDelimiter{\norm}{\lVert}{\rVert}

{\left\lbrace\begin{array}{@{}l@{}}}%
{\end{array}\right.}

\newcommand{\mi}{\mu}

\newcommand{\loc}{\mathrm{loc}}

\newcommand{\dd}{\mathrm{d}}

\newcommand{\T}{\mathbb{T}}

\newcommand{\gf}{\mathfrak{g}}
\newcommand{\hf}{\mathfrak{h}}
\newcommand{\Ls}{\mathcal{L}}
\newcommand{\Hs}{\mathcal{H}}

\usepackage{accents}

\newcommand{\Bc}{\mathcal{B}}

\newcommand{\Hc}{\mathcal{H}}

\newcommand{\Pc}{\mathcal{P}}

\newcommand{\eps}{\varepsilon}

\renewcommand{\epsilon}{\varepsilon}

\title[Schr\"odinger operators on Lie groups]{Schr\"odinger operators on Lie groups \\ with purely discrete spectrum}

\author[Tommaso Bruno]{Tommaso Bruno}
\address{Department of Mathematics: Analysis, Logic and Discrete Mathematics, Ghent University,
Krijgslaan 281, 9000 Ghent, Belgium}
\email{tommaso.bruno@ugent.be}
\author[Mattia Calzi]{Mattia Calzi}
\address{Dipartimento di Matematica, Universit\`a degli Studi di Milano, Via C.\ Saldini 50, 20133 Milano, Italy}
\email{mattia.calzi@unimi.it}


\thanks{{\em Math Subject Classification.} {Primary:} 22E30, 35J10, 58J50 {Secondary:} 47A10, 35R03}

\thanks{T.\ Bruno gratefully acknowledges support by the Research Foundation -- Flanders (FWO) through the postdoctoral grant 12ZW120N}

\DeclareMathAlphabet{\mathcal}{OMS}{cmsy}{m}{n}

\begin{document}

\begin{abstract}
On a Lie group $G$, we investigate the discreteness of the spectrum of Schr\"odinger operators of the form $\Ls +V$, where $\Ls$ is a subelliptic sub-Laplacian on $G$ and the potential $V$ is a locally integrable function which is bounded from below. We prove general necessary and sufficient conditions for arbitrary potentials, and we obtain explicit characterizations when $V$ is a polynomial on $G$ or belongs to a local Muckenhoupt class. We finally discuss how to transfer our results to weighted sub-Laplacians on $G$.
\end{abstract}

\maketitle

\section{Introduction}
The aim of this paper is to study the discreteness of the spectrum of Schr\"odinger operators on Lie groups, that is, operators of the form
\[
\mathcal{H}_V = \Ls + V,
\]
where $\Ls$ is the subelliptic sub-Laplacian on the group associated with a left-invariant sub-Riemannian structure, and $V$ is a locally integrable function which is essentially bounded from below.  When the group is a Euclidean space $\R^d$, these operators  reduce to the classical Schr\"odinger operators $\Delta + V$, where $\Delta$ is the non-negative  Laplacian. The study of their spectrum, of its discreteness in particular, has a long-standing tradition and history; see, e.g.,~\cite{Molcanov, MS, KMS, KS, MetafunePallara, Simon}, and, more recently,~\cite{Dallara}. Besides their independent interest, Schr\"odinger operators arise from the study of ``weighted'' Laplacians, namely the natural substitutes of the Laplacian when $\R^d$ is endowed with an absolutely continuous measure. The prototypical example is that of the Gaussian measure and of the Ornstein--Uhlenbeck operator.

Outside the Euclidean setting,  not much is known. Kondratiev and Shubin's influential paper~\cite{KS}, following Mol\v{c}anov~\cite{Molcanov}, considers the case of Riemannian manifolds with bounded geometry, whose nature allows however to reduce several problems at a Euclidean level.  On certain subfamilies of stratified Lie groups, some results were obtained  by Inglis~\cite{Inglis} and the authors~\cite{ BC1, BC2}. Such results, however, were mainly motivated by the study of weighted sub-Laplacians, and the potentials considered therein were related to those arising from these sub-Laplacians, for which weak results are often enough. General, yet powerful, results for Schr\"odinger operators with arbitrary potentials on Lie groups are, to the best of our knowledge, still missing. The goal of this paper is to make a first step to fill this gap. 

We consider a noncompact connected Lie group $G$, whose Lie algebra we denote with $\mathfrak{g}$ and we identify with the algebra of left-invariant vector fields on $G$.  We endow $\mathfrak{g}$ with an inner product, we fix a subspace $\Hc$ of $\mathfrak{g}$ which generates $\gf$ as a Lie algebra, and choose an orthonormal basis $(X_1,\dots, X_\nu)$ of $\Hc$. 
 We choose a left Haar measure $\mu$ on $G$ and denote with $d$ the Carnot--Carathéodory distance induced by $\Hc$; the metric measure space $(G,d, \mu)$ is then locally -- but in general not globally -- doubling. The second order differential operator
\[
\mathcal{L} = -\sum_{j=1}^\nu (X_j^2 +c_j X_j),
\]
where $ c_j =(X_j \delta_G)(e)$, $\delta_G$ is the modular function of $G$ and $e$ is the identity, is a subelliptic operator which depends only on $G$, $\mathcal{H}$ and its scalar product, and is symmetric on $L^2(\mu)$. In other words, once the left-invariant sub-Riemannian structure on $G$ has been chosen, $\Ls$ is an intrinsic operator; cf.~\cite{ABGR}. For a potential $V\in L^1_\loc(\mu)$ and bounded from below, the Hermitian form
\[
Q\colon (f,g)\mapsto \int_{G} \bigg( \sum_{j=1}^\nu (X_jf)(X_j\overline{g}) +  V f \overline{g} \bigg)\,\dd \mu
\]
induces a self-adjoint operator $\Hs_V$ on $L^2(\mu)$ which, on its domain, acts as $\Ls + V$ meant in the distributional sense. Our aim is to determine under what conditions on $V$ the spectrum of  $\Hs_V$ is purely discrete, i.e., it is a discrete set and consists of eigenvalues of finite multiplicity. 

We are interested in two types of results. On the one hand, we look for ``universal'' results, namely necessary or sufficient conditions for the discreteness of the spectrum of $\Hs_V$ with no other assumptions on $V$ than local integrability and boundedness from below. These results are usually expressed in terms of growth conditions, or rather decay, of $V$ and the associated form $Q$. Due to the high generality, however, such conditions are often hard to verify. On the other hand, then, we also look for more explicit conditions which are easier to test, for specific choices of potentials. Following~\cite{Auscher-BenAli, MetafunePallara, Dallara}, we consider the cases when $V$ is a polynomial and when $V$ is a local Muckenhoupt weight.

Inspired by results of Kondratiev and Shubin~\cite{KS} and Metafune and Pallara~\cite{MetafunePallara}, more specifically, our main ``universal'' results for the discreteness of the spectrum of $\Hs_V$ are:
\begin{itemize}
\item a characterization expressed in terms of the decay, outside compact sets, of the $L^2$-norm of functions in the ``unit ball'' of $Q$ (i.e.\ such that $(\min V+1)\norm{f}_2^2+Q(f)\leq 1$; Proposition~\ref{prop1}); 
\item a characterization in terms of the growth of the Dirichlet and Neumann eigenvalues of $\Ls+V$ on balls whose centres go to infinity (Theorem~\ref{teoKS}); 
\item a sufficient condition in terms of the thinness at infinity, in a measure-theoretic sense, of the sublevel sets of $V$ (Proposition~\ref{prop:2}).
\end{itemize}

Besides their intrinsic interest, the preceding results may be applied to the study of two specific families of potentials, namely polynomials and Muckenhoupt weights. In this case, we obtain more descriptive characterizations.
\begin{itemize}
\item If $V$ is a polynomial, we characterize the discreteness of the spectrum of $\Hs_V$ in terms of the vanishing properties of the right-invariant derivatives of $V$ (Theorem~\ref{Thm:carpolinomio} and Proposition~\ref{prop:carpolinomio});
\item if $V$ is in a local Muckenhoupt $A_p$ class, $p\in [1,\infty)$, or if $V$ is in a local $A_\infty$ Muckenhoupt class and the measure with density $V$ with respect to $\mu$ is locally doubling, we characterize the discreteness of the spectrum of $\Hs_V$  in terms of the growth of the $L^1$ norm of $V$ on balls whose centres go to infinity (Theorem~\ref{teo:Muck}).
\end{itemize}
The above results are inspired respectively by works of Metafune and Pallara~\cite{MetafunePallara} and by Auscher and Ben Ali~\cite{Auscher-BenAli} and Dall'Ara~\cite{Dallara} on $\R^d$. We emphasize, however, that the generality of our setting  requires substantially new ideas and techniques. To the best of our knowledge, they are the first results of their kind outside the Euclidean context. We also point out that the case of polynomial potentials allows us to shed some light on certain harmonic oscillators on Heisenberg groups of recent introduction, cf.~\cite{Fischer, RR1, RR2}.

Finally, we discuss and describe how Schr\"odinger operators are related to weighted sub-Laplacians  with possibly non-smooth weights, by means of a unitary equivalence at the $L^2$ level.

\smallskip

The structure of the paper is as follows. In the following Section~\ref{Sec:preliminaries}, we describe the setting in detail, we introduce all the relevant objects of the paper and fix the notation. Section~\ref{Sec:Firstcar} contains a first, fundamental characterization of the discreteness of the spectrum of $\Hs_V$, which is then applied in Section~\ref{Sec:necandsuff} to obtain universal results. Section~\ref{Sec:pol} is devoted to the study of polynomial potentials, while Section~\ref{Sec:Muck} to potentials belonging to local Muckenhoupt classes. The final Section~\ref{Sec:WSL} discusses weighted sub-Laplacians and their unitary equivalence to Schr\"odinger operators.

\subsection*{Acknowledgements}
We thank Fulvio Ricci for drawing our attention to the study of Schr\"odinger operators on Lie groups, and Gioacchino Antonelli for several discussions about polynomials and sub-Riemannian geometry. The two authors were partially supported by the GNAMPA 2020 project ``Fractional Laplacians and subLaplacians on Lie groups and trees''. Part of this research was  carried out during a visit of the two authors to the Mathematisches Forschungsinstitut Oberwolfach in June 2021, through the ``Research in Pairs'' program. We sincerely thank the Institute for the opportunity and all its staff for their kind hospitality.

\section{Lie groups and Schr\"odinger operators}\label{Sec:preliminaries}

\subsection{Preliminaries on Lie groups}

Let $G$ be a noncompact connected Lie group with identity $e$. We fix a left Haar measure $\mu$ on $G$, and denote by $\delta_G$ the modular function on $G$, defined so that $\mu(A x)=\delta_G(x) \mu(A)$ for every measurable subset $A$ of $G$ and $x\in G$. Then, $\delta_G$ is an analytic character of $G$ and the measure whose density is $\delta_G^{-1}$ with respect to $\mu$ is a right Haar measure on $G$.

We identify the Lie algebra $\mathfrak{g}$ of $G$ with the algebra of left-invariant vector fields on $G$. We denote by $\exp_G$, or simply $\exp$ if there is no risk of confusion, the exponential map from $\mathfrak{g}$ to $G$; we recall that $\exp_G$ is a local diffeomorphism at $0$. We endow $\mathfrak{g}$ with a scalar product whose induced norm we denote by $| \cdot|$. 

We fix a subspace $\Hc$ of $\mathfrak{g}$ which generates $\gf$ as a Lie algebra, and we define the associated left-invariant ``horizontal'' gradient $\nabla_\Hc$ so that $\langle \nabla_\Hc f(e), X\rangle=  (X f)(e)$ for every $f\in C^1(G)$ and $X\in \Hc$. We  identify $\Hc$ with the corresponding  left-invariant distribution on $G$ given by $ \Hc_x = \{X_x\colon X\in \Hc\}$ for $x\in G$. We also endow such distribution with the left-invariant sub-Riemannian structure induced by the scalar product on $\gf$, namely $\langle X_x, Y_x\rangle_{ \Hc_x}= \langle X,Y\rangle_\gf$ for every $X,Y\in \Hc$ and $x\in G$.
 
The distribution $\Hc$ induces a distance on $G$, called Carnot--Carathéodory distance, which can be defined as follows. For $x,y\in G$, define $\Gamma(x,y)$ as the set of all absolutely continuous curves $\gamma\colon [0,1]\to G$ such that $\gamma(0)=x$, $\gamma(1)=y$, and $\gamma'(t)\in \Hc_{\gamma(t)} $ for almost every $t\in [0,1]$. Then, define 
\[
d(x,y) = \inf \bigg\{ \int_0^1 \abs{\gamma'(t)}\,\dd t \colon \gamma \in \Gamma(x,y)\bigg\}.
\]
By the left invariance of $\Hc$, the distance $d$ is left invariant:
\[
d(xy, z) = d(y,x^{-1}z), \qquad x,y,z\in G.
\]
Equivalently, if for $x\in G$ we denote by $L_x$ the left translation operator by $x$, then $d(L_x y, L_x z)= d(y,z)$.

Note that the topology induced by $d$ is the same as the original topology of $G$ by Chow's theorem (cf.~\cite[Theorem 3.31]{AgrachevBoscainBarilari}), and that every bounded subset of $G$ is relatively compact by~\cite[Proposition 3.47 and Corollary 7.51]{AgrachevBoscainBarilari}. 

For $x\in G$ and $r>0$, we shall denote by $B(x,r)$ the open ball centred at $x$ with radius $r$ with respect to $d$. We recall, cf.~\cite{NSW}, that there is $d_0>0$ such that for all $R>0$ there is $C>0$ such that
\begin{equation}\label{pallepiccole}
C^{-1} r^{d_0}\leq \mu(B(e,r)) \leq C r^{d_0}, \qquad r\in (0,R].
\end{equation}
Observe that by the left invariance of $d$ and $\mu$, one has 
\[
\mu(B(x,r)) = \mu(B(e,r))
\]
for all $x\in G$ and $r>0$. By~\eqref{pallepiccole}, then, the metric measure space $(G,d,\mu)$ is locally doubling; that is,  for every $R>0$ there is a constant $C>0$ such that 
\begin{equation}\label{localdoubling}
\mu(B(x,2r))\leq C \mu(B(x,r)), \qquad \forall \, x\in G, \; r\in (0,R].
\end{equation}
The global doubling condition holds if and only if $G$ is of polynomial growth, see, e.g.,~\cite[Theorem II.3]{Guivarch}. 

\subsection{$L^p$ spaces and approximation procedures} We denote by $L^p(G)$ and $L^p_\loc(G)$, or simply $L^p$ and $L^p_\loc$ if there is no risk of confusion, the usual Lebesgue and local Lebesgue spaces with respect to $\mu$, respectively. We  denote by $L^2_\Hc$ the space of square integrable sections of $\Hc$;  $L^p_{\Hc}$ and $L^p_{\Hc,\loc}$ are defined similarly.  The space of smooth and compactly supported functions on $G$ will be denoted by $C_c^\infty(G)$, or simply $C_c^\infty$, and the spaces $C^\infty(G)$ and $C_c(G)$ are defined accordingly.

We now introduce some families of functions which are particularly useful in approximation procedures. We first recall that if $p,q,r\in [1,\infty]$ are given so that $\frac{1}{p}+\frac{1}{q}=1+\frac{1}{r}$, then the convolution of $f\in L^p$ and $g\in L^q$ is given by
\[
(f*g)(x) =  \int_G f(y) g(y^{-1}x)\, \dd \mu(y) = \int_G f(xy^{-1}) g(y) \delta_G^{-1}(y) \, \dd \mu(y),
\]
and it satisfies Young's inequality (see, e.g.,~\cite[Remark 2.2]{KleinRusso})
\begin{equation}\label{Young}
	\norm{f* (\delta_G^{1/p'} g)}_{r}\leq \norm{f}_{p} \norm{g}_{q},
\end{equation}	
where $p'$ denotes the conjugate index to $p$.

By an \emph{approximate identity} on $G$ we shall mean a sequence $(\eta_j)$, $j\in \N$, of positive $C^\infty_c$  functions such that
\[
\int_G \eta_j \delta_G^{-1}\,\dd \mu=1 \quad \text{and} \quad \supp{\eta_j}\subseteq B(e,1/(j+1)), \quad \text{for every } j\in \N.
\]
With this choice, $1*\eta_j=1$ and $f*\eta_j\geq0$ if $f\geq0$.  In addition, $f*\eta_j \to f$ and $\eta_j*f\to f$ in $L^p$ (resp.\ $L^p_\loc$) whenever $f\in L^p$ (resp.\ $L^p_\loc$) and $p\in [1,\infty)$. If $f\in L^p$, this follows from~\eqref{Young} and the density of $C_c$ in $L^p$. If $f\in L^p_\loc$, this follows from the convergence in $L^p$, since
	\[
	(f*\eta_j) (x)=[(f\mathbf{1}_{B(e,r+1)})*\eta_j](x), \qquad j\in \N, \: r>0,
	\]	
for every $x\in B(e,r)$, and similarly for $\eta_j*f$. Here and all throughout the paper, we denote by $\mathbf{1}_\Omega$ the characteristic function of a given measurable subset $\Omega$ of $G$.

By an \emph{approximate unity} on $G$ we shall mean a family of functions of the form $(\mathbf{1}_{B(e,{r+1})}*\psi)_{r>0}$, where $\psi \in C_c^\infty$ is positive and such that 
\[
\int_G \psi\delta_G^{-1}\,\dd \mu=1, \qquad \supp \psi \subseteq B(e,1).
\]
Then, 
\[
\mathbf{1}_{B(e,r)}\leq \mathbf{1}_{B(e,r+1)}*\psi\leq \mathbf{1}_{B(e,r+2)}, \qquad r>0,
\]
so that $(\mathbf{1}_{B(e,{r+1})}*\psi)$ is uniformly bounded and converges locally uniformly to $1$ when $r\to\infty$. Furthermore, if $X$ is a left-invariant differential operator without constant term on $G$, then the functions $X(\mathbf{1}_{B(e,r)}*\psi)=\mathbf{1}_{B(e,r)}*X\psi$ are uniformly bounded on $G$ and converge locally uniformly to $1*X\psi=0$ when $r\to +\infty$.

\subsection{Sub-Laplacians and Schr\"odinger operators}
We fix once and for all an orthonormal basis $\mathbf{X}=( X_1,\dots, X_\nu)$ of $\Hc$, so that the horizontal gradient can be written as
 \[
 \nabla_\Hc f= \sum_{j=1}^\nu (X_j f)X_j.
 \]
We denote by $\nabla_\Hc^*$ the formal adjoint of $\nabla_\Hc$ on $L^2$, and define the symmetric sub-Laplacian
\[
\mathcal{L} = \nabla_\Hc^* \nabla_\Hc=-\sum_{j=1}^\nu (X_j^2 +c_j X_j), \qquad c_j =(X_j \delta_G)(e),
\]
so that
\[
\int_ G (\mathcal{L} \varphi) \, \overline{\psi} \, \dd \mu=\int_ G \varphi\,  (\mathcal{L} \overline{\psi}) \, \dd \mu = \int_G \nabla_\Hc \varphi \cdot \nabla_\Hc \overline{\psi} \, \dd \mu
\]
for every $\varphi,\psi\in C^\infty_c$. We recall that $\Ls$ is an intrinsic operator, namely it depends only on the choice of $\Hs$ and the scalar product thereon, and that it is a sum of squares (i.e.\ has no drift) if and only if the group $G$ is unimodular; cf.~\cite{ABGR, HMM}.

Let $V$ be a function in $ L^1_\loc$ which is essentially bounded from below. Up to shifting $V$, we may assume that $V\geq 1$. Given
\[
\Dom(Q)=  \Big\{ f\in L^2\colon \nabla_\Hc f\in L^2_\Hc, \; \sqrt V f\in L^2 \Big\},
\]
we define the Hermitian form
\[
Q\colon (f,g)\mapsto \int_{G} (\nabla_\Hc f \cdot \nabla_\Hc \overline{g} +  V f \overline{g} )\,\dd \mu, \qquad  f,g\in \Dom(Q).
\]
We adopt the customary notation $Q(f) = Q(f,f)$. Since $V\geq 1$, the form $Q$ induces a norm on $\Dom(Q)$ given by $f\mapsto Q(f)^{1/2}$.

Observe that $Q$ is a positive, continuous and closed form on $L^2$; it is also densely defined, since $C^\infty_c\subseteq \Dom(Q)$. Thus, it induces a self-adjoint operator $\Hs_V$ on $L^2$ (cf., e.g.,~\cite[Definition 1.21]{Ouhabaz}) whose domain we denote by $\Dom(\Hs_V)$. Since 
\[
\langle \Hs_V f,\varphi \rangle_{L^2} = \langle \Ls f +Vf, \varphi\rangle_{L^2} 
\]
for every $f\in \Dom(\Hs_V)$ and $\varphi \in C_c^\infty$, the function $\Hs_V f$ coincides with the distribution $(\Ls+V )f$ (observe that $fV = \sqrt{V}f \sqrt{V} \in L^1_\loc$). We shall show in Lemma~\ref{lem:5} below that $C_c^\infty$ is actually dense in $\Dom(Q)$; this implies that, if $f\in \Dom(Q)$ and $(\Ls+V )f \in L^2$, then $f\in \Dom(\Hs_V)$, namely that
\[
\Dom(\Hs_V) = \{ f\in \Dom(Q)\colon (\Ls+V )f \in L^2\}.
\]

\begin{lemma}\label{lem:5}
The space $C^\infty_c$ is dense in $\Dom(Q)$.
\end{lemma}

\begin{proof}
If $f\in \Dom(Q)$ and $(\psi_r)_{r>0}$ is an approximate unity, then $f \psi_r \to f$ in $\Dom(Q)$ when $r\to \infty$. Hence, we may reduce to showing the density of $C_c^\infty$ in the space of functions in $\Dom(Q)$ with compact support.

Let now $\varphi\colon \C\to \C$ be a bounded, Lipschitz, and smooth function  which is the identity on $B(0,1)$, and define $\varphi_k= k\varphi(\,\cdot\,/k  )$ for every $k\in \N^*$. If $f \in \Dom(Q)$ is compactly supported, then $\varphi_k\circ f$ is bounded and compactly supported, belongs to $\Dom(Q)$, and converges to $f$ in $\Dom(Q)$ by dominated convergence, since  
\[
\nabla_\Hc(\varphi_k\circ f)= (\varphi'_k\circ f) \nabla_\Hc f,
\]
$\abs{\varphi_k\circ f}\leq \|\varphi\|_\infty \abs{f}$, and $\abs{\varphi'_k\circ f}\leq \|\varphi'\|_\infty $.

Therefore, one may in turn reduce to the problem of approximating an $f \in \Dom(Q)$ which is bounded and compactly supported. To do this, it is enough to consider $\eta_j*f$ where $(\eta_j)$ is an approximate identity.
\end{proof}

As mentioned above, Lemma~\ref{lem:5} implies that $\Dom(\Hs_V)$ is the set of the functions $f\in \Dom(Q)$ such that $(\mathcal{L} +V) f$, defined distributionally, belongs to $L^2$, and $\Hs_V f=\mathcal{L} f+ V f$ for every such $f$. In addition, $C^\infty_c$ is a core for $\sqrt{\Hs_V}$, but in general $C^\infty_c$ is not even contained in $\Dom(\Hc_V)$, unless $V\in L^2_\loc$.  
When $V\in L^2_\loc$, one may actually show that $C^\infty_c$ is a core for $\Hc_V$, which is therefore essentially self-adjoint on the domain $C^\infty_c$.
The proof is a verbatim adaptation of the proof of~\cite[Theorem 3.2]{BC1} (see also~\cite[Theorem X.28]{ReedSimon}), which is based on Kato's inequality for $\Ls$, see~\cite[Theorem X.27]{ReedSimon}, by using approximate unities as defined above.

\begin{proposition}
If $V\in L^2_\loc$, then $C^\infty_c$ is a core for $\Hc_V$.
\end{proposition}

\section{A first characterization}\label{Sec:Firstcar}
Our first result is inspired by~\cite{KS}, and concerns a characterization of the discreteness of the spectrum of $\Hs_V$ in terms of the decay of the $L^2$-norm of functions in $\Dom(Q)$ outside compact sets. It may be thought of as one of the cornerstones of the paper. In order to prove it, we need a few lemmas. 

The first one is elementary. Its proof is an easy adaptation of~\cite[Theorems 2.32 and 2.33]{AF}, but for the ease of the reader we provide all the details.

\begin{lemma}\label{lem:2}
Suppose $p\in [1,\infty)$. Then a subset  $A \subset L^p$ is relatively compact if and only if $A_r = \{ \mathbf{1}_{B(e,r)} f\colon f\in A\}$ is relatively compact for all $r>0$, and 
\begin{equation}\label{AKcompl}
	\lim_{r\to \infty} \sup_{f\in A} \int_{G\setminus B(e,r)} \abs{f}^p\,\dd \mu=0.
\end{equation}
\end{lemma}

\begin{proof}
Assume that $A$ is relatively compact, and fix $\epsilon>0$. Then, there is a finite subset $F$ of $A$ such that 
	\[
	A\subseteq \bigcup_{f\in F} B(f,\eps/2),
	\]
where the balls are meant with respect to the $L^p$ norm. In other words, for all $f\in A$ there is $f_0\in F$ such that $\norm{f-f_0}_p<\eps/2$. Moreover, there is $r>0$ such that $\norm{(1-\mathbf{1}_{B(e,r)}) f_1}_p<\eps/2$ for all $f_1\in F$. Hence 
	\[
\|f (1- \mathbf{1}_{B(e,r)})\|_p\leq \|f_0(1-\mathbf{1}_{B(e,r)}) \|_p+\| (f-f_0) (1-\mathbf{1}_{B(e,r)})\|_p<\eps.
	\]
This proves~\eqref{AKcompl}. Since the linear map $f\mapsto \mathbf{1}_{B(e,r)} f$ is continuous from $L^p$ into itself, $ A_r$ is relatively compact for every $r>0$, and the ``only if'' part is proved.
	
	Conversely, assume that $A_r$ is relatively compact for every $r>0$ and that~\eqref{AKcompl} holds. Fix $\eps>0$, and choose $r>0$ so that $\norm{(1-\mathbf{1}_{B(e,r)})f}_p<\eps/3$ for every $f\in A$. In addition, choose a finite subset $F$ of $A$ such that
	\[
	A_r \subseteq  \bigcup_{f\in F_r}B(f,\eps/3), \qquad F_r = \{ \mathbf{1}_{B(e,r)} f\colon f\in F\}.
	\]
Then $A\subseteq \bigcup_{f\in F} B(f,\eps)$, whence $A$ is relatively compact.
\end{proof}

For the next lemma we shall need Sobolev spaces on $G$. Given $\alpha\geq 0$ and $p\in (1,\infty )$, define 
\[
L^p_\alpha = \{ f\in L^p\colon \Ls^{\alpha/2} f\in L^p\},
\]
endowed with the norm $f\mapsto \norm{f}_p+\norm{\mathcal{L}^{\alpha/2} f}_p$. These spaces have been introduced and studied in~\cite{BPTV}.  In particular they interpolate well with respect to the complex method, see~\cite[Lemma 3.1]{BPTV}, and by~\cite[Theorem 4.4]{BPTV}, given $\alpha_1,\alpha_2\geq0$ and $p_1,p_2\in (1,\infty)$,
\begin{equation}\label{Sobolevembedding}
L^{p_1}_{\alpha_1} \subseteq  L_{\alpha_2}^{p_2}  \qquad   \mbox{if} \quad \frac{1}{p_1}-\frac{\alpha_1-\alpha_2}{d_0}\leq \frac{1}{p_2}\leq \frac{1}{p_1},
\end{equation}
where we recall that $d_0$ is the local dimension of $G$ determined by~\eqref{pallepiccole}.
The following lemma is the generalization of a result in~\cite{FollandSE} to the case of general Lie groups. 

\begin{lemma}\label{lem:3}
Let $\varphi \in C^\infty_c$ be given. Then the linear map
\[
T_\varphi \colon L_{\alpha_1}^{p_1} \to L^{p_2}, \qquad T_\varphi f =f \varphi,
\]
is compact for all $p_1,p_2\in (1,\infty)$ and  $\alpha_1\geq 0$ such that $0\leq \frac{1}{p_1}-\frac{1}{p_2}<\frac{\alpha_1}{d_0}$.
\end{lemma}

\begin{proof}
	First of all we notice that, by means of a finite partition of the unity, we may reduce to the case when $\varphi$ is supported in the domain of a local chart $(U,\psi)$ with $U$ being relatively compact. Now, define
	\[
	\mathcal{H}^{(1)}=\mathcal{H}, \qquad \mathcal{H}^{(k+1)} =\mathcal{H}^{(k)} + [\mathcal{H}^{(1)},\mathcal{H}^{(k)}], \qquad k\geq  1,
	\]
	so that $(\mathcal{H}^{(k)})$ is an increasing sequence of subspaces of $\gf$ whose union is $\gf$ by the bracket-generating property of $\mathcal{H}$.
	
	Let $m$ be the smallest integer such that $\mathcal{H}^{(m)}=\gf$. Then, every left-invariant differential operator of order at most $k$ on $G$ can be decomposed as a finite linear combination of differential operators of the form $X_{j_1}\cdots X_{j_\ell}$, where $\ell\leq km$ and $j_1,\dots, j_\ell\in \{1,\dots,\nu\}$. Then, the map $f\mapsto (T_\varphi f) \circ \psi^{-1}$ maps $L_{k m}^{p_1}$ continuously into $W^{p_1,k}(\psi(U))$ for every $k\in \N$, the latter being the Euclidean Sobolev space on $\psi(U)$ (cf.~\cite[Proposition 3.3]{BPTV}). Since we may assume that $\psi(U)$ is a ball of some Euclidean space, the classical Rellich--Kondrakov theorem implies that $T_\varphi\colon L_{k m}^{p_1}\to L^{p_1}$ is compact if $k$ is large enough.
	
	The result follows now by interpolation. Indeed, by~\cite[Lemma 3.1]{BPTV}, for all $\alpha_1>0$ the space $L^{p_1}_{\alpha_1}$ is intermediate between $L_{k m}^{p_1}$ and $L^{p_1}$ (up to taking $k$ large enough). Hence, since trivially $T_\varphi \colon L^{p_1} \to L^{p_1}$, by~\cite[Theorems 3.8.1 and 4.7.1]{BerghLofstrom} one gets that $T_\varphi$ induces a compact linear map from $L_{\alpha_1}^{p_1}$ to $L^{p_1}$. Hence, the statement is proved for $p_2=p_1$.

Let now $\alpha_1>0$ and $p_2>p_1$ be such that $\frac{1}{p_1}-\frac{1}{p_2}<\frac{\alpha_1}{d_0}$. Choose $\beta_1<\alpha_1$ such that 
\[
\frac{1}{p_1}-\frac{1}{p_2}<\frac{1}{d_0} (\alpha_1 - \beta_1).
\]
Then, by combining  the  map $T_\varphi \colon L^{p_2}_{\beta_1} \to L^{p_2}$ which is compact by the previous step, with the continuous embedding $L^{p_1}_{\alpha_1} \subseteq   L^{p_2}_{\beta_1}$ given by~\eqref{Sobolevembedding}, one gets that $T_\varphi$ induces a compact map from $L^{p_1}_{\alpha_1} $ to $L^{p_2}$.
\end{proof}

We are now in a position to state the first aforementioned characterization of the discreteness of the spectrum of $\Hs_V$. It will be a fundamental criterion for the remainder of the paper.

\begin{proposition}\label{prop1}
	The operator $\Hs_V$ has purely discrete spectrum if and only if 
	\[
	\lim_{r\to \infty} \sup_{\substack{ f\in \Dom(Q) \\ Q(f)\leq 1}} \int_{G\setminus B(e,r)} \abs{f}^2\,\dd \mu=0.
	\]
\end{proposition}

\begin{proof}
	Observe first that $\Hs_V$ has purely discrete spectrum if and only if $\sqrt{\Hs_V}$ has purely discrete spectrum. This holds if and only if the (bounded) inverse $\sqrt{\Hs_V}\,^{-1}$  of $\sqrt{\Hs_V}$, which exists since $\Hs_V \geq 1$, is compact; cf.~\cite[Theorem 11.3.13]{Ol}. This is, in turn, equivalent to saying that the closed subset 
	\[
	A=\{f\in \Dom(Q)\colon Q(f)\leq 1\}
	\]
	of $L^2$ is compact. Observe that $ A$ is contained in the unit ball of $L_1^{2}$ (cf.~\cite[Proposition 3.3]{BPTV}), so that for every $\varphi \in C^\infty_c$ the set $\varphi A = \{ \varphi f \colon f\in A\}$ is precompact in $L^2$ by Lemma~\ref{lem:3}. The statement then follows from Lemma~\ref{lem:2}.
\end{proof}

\begin{remark}
The results of this section can be extended to the case of general relatively invariant measures, with minor modifications. Consider a continuous positive character $\chi$ of $G$, the measure $\mu_\chi$ with density $\chi\delta^{-1}_G$ with respect to $\mu$, and the sub-Laplacian
\begin{equation}\label{deltachi}
\Delta_\chi = -\sum_{j=1}^\nu (X_j^2 + c_j X_j), \qquad c_j = (X_j \chi)(e).
\end{equation}
Then $\Delta_\chi $ is essentially self-adjoint on $L^2(\mu_\chi)$, and all the sub-Laplacians with drift which are symmetric on $L^2(\eta)$ for some measure $\eta$ are of the form~\eqref{deltachi}, with $\eta=\mi_\chi$; cf.~\cite{HMM, BPTV}. One can consider the form 
\[
Q_\chi \colon (f,g)\mapsto \int_{G} (\nabla_\Hc f \cdot \nabla_\Hc \overline{g} +  V f \overline{g} )\,\dd \mu_\chi
\]
with its natural domain, and obtain a Schr\"odinger operator $\mathcal{H}_{V,\chi}$ which coincides, on its domain, with the operator $\Delta_\chi + V$ meant in the distributional sense. When $\chi= \delta_G$, one obtains the form $Q$ and the Schr\"odinger operator $\mathcal{H}_V$. All the results of this section hold with the obvious modifications (and the same proofs) in this more general setting; we shall not go into details here, however, as they would be an unnecessary complication. All the results from the next section on will indeed require a left-invariant measure.
\end{remark}

\section{General necessary and sufficient conditions}\label{Sec:necandsuff}

\subsection{Dirichlet and Neumann spectra}
Following~\cite{KS}, for all non-empty open subsets $U$ of $G$ we define the bottoms of the Dirichlet and Neumann spectra of $\Hs_V$, respectively, as
\begin{align*}
\sigma_D(U)& =\inf\bigg\{\int_U \big(\abs{\nabla_\Hc f}^2+ V \abs{f}^2\big)\,\dd \mu \colon f\in C^\infty_c(U), \; \norm{f}_2=1  \bigg\},\\
\sigma_N(U) &= \inf\bigg\{\int_U \big(\abs{\nabla_\Hc f}^2+ V \abs{f}^2\big)\,\dd \mu\colon f\in C^\infty(G), \; \norm{\mathbf{1}_Uf}_2=1\bigg\}.
\end{align*}
It immediately follows from the definition that $\sigma_N(U)\leq \sigma_D(U)$. The aim of this section is to characterize the discreteness of the spectrum of $\Hs_V$ in terms of the behaviour of $\sigma_D(B)$ and $\sigma_N(B)$ when $B$ is a ball; see Theorem~\ref{teoKS} below. We begin with a partial result.

\begin{proposition}\label{prop:3}
	If $\Hs_V$ has purely discrete spectrum, then $\sigma_D(B(x,r))\to \infty$ as $x\to \infty$, for all $r>0$. Conversely, if there is $r>0$ such that $\sigma_N(B(x,r))\to \infty$ as $x\to \infty$, then $\Hs_V$ has purely discrete spectrum.
\end{proposition}

\begin{proof}
	Assume that $\Hs_V$ has purely discrete spectrum. Fix $\eps>0$, and observe that by Proposition~\ref{prop1}  there is $R>0$ such that 
	\[
	\int_{G\setminus B(e,R)} \abs{f}^2\,\dd \mu < \eps
	\]
for all $f\in \Dom(Q)$ such that $Q(f)\leq 1$. Therefore, if $x\in G\setminus B(e,R+r)$ and $f\in C^\infty_c(B(x,r))$ satisfies $\norm{f}_2=1$, we have $Q(f)\geq\eps^{-1} $, whence $\sigma_D(B(x,r))\geq \eps^{-1}$. The conclusion follows by the arbitrariness of $\eps$.
	
Conversely, assume that there is $r>0$ such that $\sigma_N(B(x,r))\to \infty$ as $x\to \infty$. By~\cite[Lemma 1]{Anker}, see also~\cite[Lemma 2.3]{B1}, there are $n\in\N$ and a countable subset $\mathfrak{U}$ of $G$ such that 
\[
G=\bigcup_{x\in \mathfrak{U}} B(x,r)
\] 
and such that, for all $x_0\in \mathfrak{U}$, the intersection $B(x_0,r) \cap B(x,r)$ is non-empty for at most $n$ elements $x \in \mathfrak{U}$.
	
For $R>0$, define 
\[
\widetilde \sigma_R = \inf \{ \sigma_N(B(x,r))\colon x\in \mathfrak{U},\; \; x\not \in B(e,R+r) \},
\]
so that,  by assumption, $\widetilde \sigma_R\to \infty$ as $R\to\infty$. Then, define 
\[
I_R = \{ x \in \mathfrak{U}\colon  x\not \in  B(e,R +r)\}.
\]
If $f\in C^\infty_c$, then
\begin{align*}
	\int_{G\setminus B(e,R)} (\abs{\nabla_\Hc f}^2+ V \abs{f}^2)\,\dd \mu
	&\geq \frac{1}{n}\sum_{x\in I_R} \int_{B(x,r)} (\abs{\nabla_\Hc f}^2+ V \abs{f}^2)\,\dd \mu\\
		&\geq\frac{\widetilde \sigma_R}{n}\sum_{x\in I_R} \int_{ B(x,r)} \abs{f}^2\,\dd \mu \\
		&\geq  \frac{\widetilde \sigma_R}{n} \int_{G\setminus B(e,R+2r)} \abs{f}^2\,\dd \mu.
\end{align*}
	Therefore,
	\[
	\int_{G\setminus B(e,R+2r)} \abs{f}^2\,\dd \mu \leq \frac{n}{\widetilde \sigma_R} Q(f)
	\]
	for every $f\in  C^\infty_c$. Since $C^\infty_c$ is dense in $\Dom(Q)$ by Lemma~\ref{lem:5}, it follows that 
	\[
	\lim_{R\to \infty} \sup_{Q(f)\leq 1}\int_{G\setminus B(e,R+2r)} \abs{f}^2\,\dd \mu=0,
	\]
	hence $\mathcal{H}_V$ has purely discrete spectrum by Proposition~\ref{prop1}.
\end{proof}

\begin{corollary}\label{corprop:3}
If $\Hs_V$ has purely discrete spectrum, then
\[
\lim_{x\to \infty} \int_{B(x,r)} V\, \dd \mu = \infty 
\]
for all $r>0$.
\end{corollary}
\begin{proof}
Fix $r>0$ and $\psi \in C_c^\infty(B(e,r))$ so that $\| \psi\|_2=1$. For every $x\in G$, define $\psi_x =L_x\psi$ and observe that $\| \psi_x\|_2= 1$ for all $x\in G$. Then
\begin{align*}
\sigma_D(B(x,r)) 
&\leq \int_{B(x,r)} (|\nabla_\Hc \psi_x|^2 + V|\psi_x|^2) \, \dd \mu  \leq C\bigg( 1+  \int_{B(x,r)} V\, \dd \mu\bigg)
\end{align*}
where $C\coloneqq \max( \|\nabla_\Hc \psi\|_2^2, \|\psi\|_\infty^2)$ is independent of $x$. The conclusion follows, since $\sigma_D(B(x,r))  \to \infty$ for $x\to \infty$ by Proposition~\ref{prop:3}.
\end{proof}

We shall now proceed to show  that $\sigma_D(B(x,r)) \to \infty$ when $x\to \infty$ if and only if the same holds for $\sigma_N(B(x,r))$. This will refine Proposition~\ref{prop:3} into a characterization  of the discreteness of the spectrum of $\Hc_V$, which is Theorem~\ref{teoKS} below. For notational convenience, given a measurable subset $U$ of $G$ such that $0<\mi(U)<\infty$,  we define the mean of $f\in L^1(U)$ as
\[
\dashint_{U} f \,\dd \mu = \frac{1}{\mu(U)} \int_U f \,\dd \mu.
\]
We will also need the Poincaré inequality on $G$, see~\cite[Theorem 3.1]{BPV1}. Given $p\in (1,\infty)$ and $R>0$,  there exists a positive constant $C$ such that
\begin{equation}\label{Poincare}
\dashint_{B\times B} \abs{f(x)-f(y)}^p\,\dd \mu(x)\, \dd \mu(y)\leq C r^p \dashint_{B} \abs{\nabla_\Hc f}^p\,\dd \mu
\end{equation}
for all balls $B$ of radius $r\in (0,R]$  and $f\in C^\infty$. For $x\in G$ and $r>0$, we shall denote by $S_{x,r}$ the space
\begin{equation}\label{Sxr}
S_{x,r} = \bigg\{ f\in C^\infty(G) \colon \int_{B(x,r)}\abs{\nabla_\Hc f}^2\,\dd\mi =1\bigg\}.
\end{equation}
\begin{lemma}\label{prop:4}
Let $x\in G$ and $r>0$ be given. If $(f_k)$ is a sequence in $S_{x,r}$, then
	\[
	\bigg( f_k-\dashint_{B(x,r)}f_k\,\dd \mi \bigg)_{k\in \N}
	\]
	has a convergent subsequence in $L^2(B(x,r))$.
\end{lemma}

\begin{proof}
Combining the Poincaré inequality~\eqref{Poincare} with the local doubling condition~\eqref{localdoubling} and with~\cite[Theorem 9.7 and Corollary 9.5]{HK}, we find $q>2$ and $C>0$ such that, for all $f\in C^\infty$,
	\[
\bigg(\dashint_{B(x,r)} |f - \tilde f|^q \, \dd \mu \bigg)^{1/q} \leq C r \bigg( \dashint_{B(x,r)} |\nabla_\Hc f|^2\, \dd \mu \bigg)^{1/2},
	\]
	where $\tilde f\coloneqq\dashint_{B(x,r)}f\,\dd \mi $. Hence, the conclusion follows by~\cite[Theorem 8.1 and 9.7]{HK}, again by the Poincaré inequality~\eqref{Poincare}.
\end{proof}

\begin{lemma}\label{lem:1}
Suppose $r>0$ and define
\begin{equation}\label{condC}
C(r) = \frac{4 \, \mi(B(e,r))}{\mi(B(e,r/2))} +1.
\end{equation}
Then, there is a decreasing function $\omega_{r} \colon (0,1) \to (0,\infty]$ such that, for all $x\in G$, $t\in (0,1)$ and $f\in C^\infty$,
\begin{equation}\label{eqlemmaKS}
 \int_{B(x,r)} \abs{f}^2\,\dd \mi \leq C(r) \int_{B(x,t r)} \abs{f}^2\,\dd \mi + \omega_r(t) \int_{B(x,r)} \abs{\nabla_\Hc f}^2\,\dd \mi,
\end{equation}
and such that $\lim\limits_{t\to 1^-} \omega_{r}(t)=0$.
\end{lemma}

When $G$ is endowed with a \emph{Riemannian} structure, it is possible to prove that $\omega_{r}(t)=O((1-t)^2)$ for $t\to 1^-$, at least when $r$ is sufficiently small (cf.~\cite[Lemma 2.8]{KS}).  Since the proof uses a good parametrization of the geodesics (namely, the exponential map, for small balls), however, it is not clear to us whether it can be extended to a sub-Riemannian setting. We shall then adopt a different strategy from that of~\cite{KS}.

\begin{proof}
	For every $t\in (0,1)$,  define
	\[
	\omega_{x,r}(t)= \sup_{f \in S_{x,r}} \bigg( \int_{B(x,r)} \abs{f}^2\,\dd \mi-C(r) \int_{B(x,t r)} \abs{f}^2 \,\dd\mi \bigg),
	\]
	where $S_{x,r}$ was given in~\eqref{Sxr}.  Observe that $\omega_{x,r} = \omega_{e,r}$ for all $x\in G$ by left invariance, and that $\omega_{e,r}$ is decreasing on $(0,1)$.  We then define $\omega_r = \omega_{e,r}$, so that~\eqref{eqlemmaKS} holds.
	
	By taking a non-zero $f\in C^\infty_c(B(e,r) \setminus \overline B(e,t r) )$  with normalized $L^2$ norm of the gradient, one also sees that $\omega_{r}(t) > 0$ for every $t\in (0,1)$, so that $\omega_{r} \colon (0,1) \to (0,\infty]$. It remains only to prove that  $\omega_{r}(t) \to 0$ when $t\to 1^-$. For notational convenience, we shall just write $\omega$ and $C$ in place of $\omega_r$ and $C(r)$ in the remainder of the proof.
	
Assume by contradiction that $\omega(t)$ does not tend to $0$ when $t \to 1^-$, and fix a strictly increasing sequence $(t_k)$ converging to $1$ such that $t_0\geq 1/2$. Then, there are $\eps>0$ and a sequence $(f_k)$ in $S_{e,r}$ such that
	\begin{equation}\label{eq:1}
	\int_{B(e,r)} \abs{f_k}^2\,\dd \mi\geq \eps+C \int_{B(e, t_k r)} \abs{f_k}^2 \, \dd \mu
	\end{equation}
	for every $k\in\N$. 
	
	By Lemma~\ref{prop:4}, we may assume that the sequence $(f_k-\tilde f_k)$ converges to some $f$ in $L^2(B(e,r))$,  where
	\[
\tilde f_k\coloneqq \dashint_{B(e,r)} f_k\,\dd \mi.
	\]
Since $(f_k-\tilde f_k)$ converges to $f$ also in $L^1(B(e,r))$, one has $\int_{B(e,r)} f\,\dd \mi=0$. We now distinguish two cases, depending on the behaviour of $(\tilde f_k)$.
	
	Assume first that the sequence $(\tilde f_k)$ in $\C$ has some bounded subsequence. Up to passing to a subsequence, we may therefore assume that $(\tilde f_k)$ converges to some $c_0\in \C$, so that $(f_k)$ converges to $f+c_0$ in $L^2(B(e,r))$. Hence, taking the limit in~\eqref{eq:1},
	\[
	\int_{B(x,r)} \abs{f+c_0}^2\,\dd \mi\geq \eps+C \int_{B(x, r)} \abs{f+c_0}^2 ,
	\]
	which is a contradiction, since $\eps>0$ and $C\geq 1$.
	
Assume then that $(\tilde f_k)\to \infty$. By~\eqref{eq:1}, for all $k\in \N$
	\begin{equation}\label{eq:2}
	\begin{split}
	\bigg(\dashint_{B(e,r)} \abs{f_k-\tilde f_k}^2\,\dd \mi\bigg)^{1/2}&\geq \bigg(\dashint_{B(e,r)} \abs{f_k}^2\,\dd \mi\bigg)^{1/2}-\abs{\tilde f_k}\\
		&\geq \eps'+C' \bigg( \dashint_{B(e,t_k r)} \abs{f_k}^2\,\dd \mi\bigg) ^{1/2}-\abs{\tilde f_k},
	\end{split}
	\end{equation}
where  we set 
	\[
	\eps' = \bigg( \frac{ \eps}{ 2 \mi(B(e,r))}\bigg) ^{1/2}, \qquad C' =\bigg( \frac{ C \mi(B(e, r/2)) }{ 2 \mi(B(e,r))}\bigg) ^{1/2}.
	\]
Observe moreover that 
\begin{align*}
	\dashint_{B(e,t_k r)} \abs{f_k}^2\,\dd \mi&= \dashint_{B(e,t_k r)} \abs{f_k-\tilde f_k}^2\,\dd \mi+\abs{\tilde f_k}^2+ 2\,\Re\bigg(  \overline{\tilde f_k}\dashint_{B(e, t_k r)} (f-\tilde f_k)\,\dd \mi\bigg) \\
		&\geq \dashint_{B(e,t_k r)} \abs{f_k-\tilde f_k}^2\,\dd \mi+\abs{\tilde f_k}^2-2 \abs{\tilde f_k}\bigg|\dashint_{B(e, t_k r)} (f-\tilde f_k)\,\dd \mi \bigg|.
\end{align*}
This together with~\eqref{eq:2} implies that, for $k\in \N$,
\begin{equation}\label{secondcontrad}
\begin{split}
		&\bigg(\dashint_{B(e,r)} \abs{f_k-\tilde f_k}^2\,\dd \mi\bigg)^{1/2}\geq \eps'+C'  \bigg( \frac 1 2\dashint_{B(e,t_k r)} \abs{f_k-\tilde f_k}^2\,\dd \mi\bigg) ^{1/2} \\
			&\hspace{2.8cm}  + \abs{\tilde f_k}\bigg( \frac{C'}{\sqrt 2}-1\bigg)- \sqrt{2}\,C'\bigg( \abs{\tilde f_k}\bigg|\dashint_{B(e, t_k r)} (f-\tilde f_k)\,\dd \mi\bigg|\bigg)^{1/2}.
			\end{split}
\end{equation}
For $k\to \infty$, since
\[
\dashint_{B(e,t_k r)} \abs{f_k-\tilde f_k}^2\,\dd \mi \to \dashint_{B(e,r)}\abs{f}^2\,\dd \mi, \qquad \dashint_{B(e, t_k r)} (f-\tilde f_k)\,\dd \mi \to 0,
\]
and $C'>\sqrt 2$ by~\eqref{condC}, the right hand side in~\eqref{secondcontrad} goes to $+\infty$ while the left hand side remains bounded. This gives a contradiction and completes the proof.
\end{proof}

\begin{theorem}\label{teoKS}
Let $r>0$ be given. Then, the following conditions are equivalent:
	\begin{enumerate}
		\item[\textnormal{(1)}] $\Hc_V$ has purely discrete spectrum;
		
		\item[\textnormal{(2)}] $\lim\limits_{x\to \infty}\sigma_D(B(x,r))=+\infty$;
		
		\item[\textnormal{(3)}] $\lim\limits_{x\to \infty}\sigma_N(B(x,r))=+\infty$. 
	\end{enumerate}
\end{theorem}

\begin{proof}
To begin with, we observe that the statements (1) $\implies$ (2) and (3) $\implies$ (1) follow from Proposition~\ref{prop:3}. Thus, we only need to prove that (2) implies (3).
	
Let $C=C(r)$ be as in~\eqref{condC} and $\omega_{r}$ as in Lemma~\ref{lem:1}. Fix $t_0\in (0,1)$ such that $\omega_{{r}}(t_0)<\infty$, so that $\omega_{{r}}(t) \in (0, \omega_{{r}}(t_0)]$ for all $t\in [t_0,1)$. Choose a  decreasing homeomorphism $\omega\colon [t_0,1]\to [0,\infty]$ such that $\omega_{r}\leq \omega$ on $[t_0,1)$.

For $t>0$,  pick a positive $\eta_t\in C^\infty_c$ such that 
\[
\int_G \eta_t \delta_G^{-1}\,\dd \mi=1, \qquad \supp{\eta_t}\subseteq B(e,t).
\]
We then select a strictly increasing sequence $(t_k)$,  $k\geq 1$, such that $t_1>t_0$ and $t_k \to 1$. For all $k\in \N$ and $t\in (t_k, t_{k+1}]$, we define 
	\[
	\tau_{t} =\mathbf{1}_{B(x, \frac{r}{2}(1+t_{k+1}))}* \eta_{\frac{r}{2}(1-t_{k+1})}.
	\]
Then  for $t\in (t_0,1)$
	\[
	\tau_t\in C^\infty_c(B(x,r)) , \qquad \mathbf{1}_{B(x, t r)}\leq \tau_t\leq \mathbf{1}_{B(x,r)}.
	\]
	In addition, by defining
\[
\alpha \colon (t_0,1) \to (0,\infty), \qquad \alpha(t) = \sup \{ \| \nabla_\Hc \tau_s \|_\infty \colon s\in (t_0,t]\} + \frac{1}{1-t},
\]
one gets 
\[
|\nabla_\Hc \tau_t| \leq \alpha(t) \qquad \forall t\in (t_0,1).
\]
Notice that  $\alpha(t)$ is  strictly increasing because of the extra term $1/(1-t)$.
 
By  definition of $\sigma_N(B(x,r))$, there is $f\in C^\infty$ such that
\begin{equation}\label{propf}
\int_{B(x,r)} \abs{f}^2\,\dd \mi=1, \qquad \text{and}\qquad	Q(f)\leq \sigma_N (B(x,r))+1.
\end{equation} 
	Then, for every $t\in (t_0,1)$,
	\begin{align*}
|\nabla_\Hc(f\tau_{t})|^2 + V|f\tau_{t}|^2 
& \leq |\nabla_\Hc f|^2 + V|f|^2 + \abs{f }^2\abs{\nabla_\Hc \tau_{t}}^2+ 2 \abs{\tau_{t} f }\abs{\nabla_\Hc f }\abs{\nabla_\Hc\tau_{t}})\\
& \leq  2 (|\nabla_\Hc f|^2 + V|f|^2 + |f|^2 |\nabla_\Hc \tau_{t}|^2).
	\end{align*}
	Therefore, by integrating over $B(x,r)$ and using~\eqref{propf},
\begin{equation}\label{fromDtoN}
	Q(f \tau_{t})\leq 2 Q(f ) +2\alpha^2(t) \leq  2\sigma_N(B(x,r))+2+2\alpha^2(t).
\end{equation}
	Now, define 
	\[
	t_x = \omega^{-1}\bigg(\frac{1}{2 \sigma_N(B(x,r))+2}  \bigg),
	\]
	and observe that,  by the choice of $\omega$ and of $f$ as in~\eqref{propf},
	\[
	\omega_r(t_x)\int_{B(x,r)} \abs{\nabla_\Hc f}^2\,\dd \mi \leq \omega(t_x)Q(f ) \leq \omega(t_x) (\sigma_N(B(x,r))+1) =\frac 1 2.
	\]
	Hence, by the definition of $\omega_r$, see Lemma~\ref{lem:1},
	\[
	\begin{split}
	1=\int_{B(x,r)}\abs{f }^2\,\dd \mi
		&\leq C \int_{B(x,t_x r)}\abs{f }^2\,\dd \mi+\omega_{r}(t_x) \int_{B(x,r)} \abs{\nabla_\Hc f }^2\,\dd \mi\\
		&\leq C \int_{B(x,t_x r)}\abs{f }^2\,\dd \mi+\frac 1 2.
	\end{split}
	\]
From this last inequality and the fact that $\tau_{t_x}=1$ on $B(x,t_x r)$, we deduce that
\begin{equation}\label{normalizingf}
	\int_{B(x,r)} \abs{f \tau_{t_x}}^2\,\dd \mi\geq \int_{B(x,t_x r)}\abs{f }^2\,\dd \mi\geq \frac{1}{2 C}.
\end{equation}
	Therefore, from~\eqref{normalizingf} and~\eqref{fromDtoN}, we get
\begin{align*}
\sigma_D(B(x,r))
&\leq 2C Q(f\tau_{t_x})\\
& \leq 4 C\bigg[\sigma_N(B(x,r))+1+\alpha^2\bigg( \omega^{-1} \bigg(\frac{1}{2 \sigma_N(B(x,r))+2} \bigg)\bigg)\bigg]
\end{align*}
	for all $x\in G$. Observe now that the function $h\colon \R_+\to \R_+$ defined by
	\[
	h(s) = s+1+ \alpha^2 \bigg(\omega^{-1}\bigg( \frac{1}{2s+2}\bigg)\bigg)
	\]
	is well defined, continuous, and strictly increasing, since $\alpha$ is strictly increasing while both $\omega^{-1}$ and $s\mapsto 1/(2s+2)$ are strictly decreasing. In addition, since $h(s) \to\infty$ as $s\to \infty$, also $h^{-1}(s)\to \infty$ when $s\to \infty$. Hence,
	\[
	\liminf_{x\to \infty} \sigma_N(B(x,r))\geq \lim_{x\to \infty} h^{-1}\bigg(  \frac{\sigma_D(B(x,r))}{4 C}\bigg)=+\infty,
	\]
	whence (3).
\end{proof}

\subsection{Thin potentials}

The following sufficient condition is inspired by~\cite[Theorem 3.1]{MetafunePallara}. As an application, see Corollary~\ref{cormagrpol} below, we obtain a generalization of a theorem of Simon in the Euclidean setting concerning ``polynomially thin'' potentials~\cite[Theorem 2]{Simon}; see also~\cite[Theorem 4.2]{BC1} on stratified Lie groups.

\begin{proposition}\label{prop:2}
If there is $r>0$ such that, for every $M>0$,
	\[
	\lim_{x\to \infty} \mu\big( \{ y\in G\colon V(y) \leq M\} \cap B(x,r) \big)=0,
	\]
then $\Hs_V$ has purely discrete spectrum.
\end{proposition}

\begin{proof}
	Given $M>0$, set $\Omega_M= \{ y\in G\colon V(y)\leq M\}$.	For every $f\in \Dom(Q)$ and $x\in G$,
\begin{equation}\label{eq1}
	\int_{B(x,r)\cap (G\setminus \Omega_M)} \abs{f}^2\leq \frac 1 M \int_{B(x,r)} V \abs{f}^2\,\dd \mu.
\end{equation}
Let $\psi \in C_c^\infty$ be such that $\mathbf{1}_{B(e,r)}\leq \psi\leq \mathbf{1}_{B(e,2r)}$, and for $x\in G$ consider its translate $\psi_x = L_x \psi$. By the left invariance of the measure, the Sobolev embeddings~\eqref{Sobolevembedding} and~\cite[Proposition 3.3]{BPTV}, there are $p>2$ and two constants $C_1,C_1'>0$ such that
\begin{align*}
\bigg( \int_{B(x,r)} \abs{f}^p\,\dd \mu\bigg)^{1/p}  \leq \|\psi_x f\|_p & \leq C_1' \|\psi_xf\|_{L^2_1}\\
&\leq C_1\bigg( \int_{B(x,2r)} (\abs{f}^2+\abs{\nabla_\Hc (f \psi_x)}^2) \,\dd \mu \bigg)^{1/2}.
\end{align*}
Since
\[
|\nabla_\Hc (f \psi_x)| \leq |\nabla_\Hc f|  + |\nabla_\Hc \psi_x| |f| \leq C_2 ( |\nabla_\Hc f|  +   |f|),
\]
one gets that there is $C>0$ such that
	\[
	\bigg( \int_{B(x,r)} \abs{f}^p\,\dd \mu\bigg)^{1/p}\leq C \bigg( \int_{B(x,2r)} (\abs{ f}^2+\abs{\nabla_\Hc f}^2) \,\dd \mu \bigg)^{1/2}
	\]
	for all $x\in G$ and $f\in \Dom(Q)$. Hence, by H\"older's inequality, 
\begin{equation}\label{eq2}
	\int_{B(x,r)\cap \Omega_M} \abs{f}^2\,\dd \mu\leq C^2 \mu(\Omega_M\cap B(x,r))^{1-2/p}  \int_{B(x,2r)} (\abs{f}^2+\abs{\nabla_\Hc f}^2) \,\dd \mu.
\end{equation}
Fix now $\eps>0$. By assumption, there exists $R>0$ such that 
\[
C^2 \mu(\Omega_M\cap B(x,r))^{1-2/p}\leq\epsilon
\]
if $d(x,e)>R$;  take $M\geq  \eps^{-1}$. Then, by~\eqref{eq1} and~\eqref{eq2}, 
\begin{equation}\label{eq3}
	\int_{B(x,r)}\abs{f}^2\,\dd \mu\leq \eps \int_{B(x,2r)} ( \abs{\nabla_\Hc f}^2+ (1+V)\abs{f}^2  )\,\dd \mu, \qquad d(x,e)>R. 
\end{equation}
We use again a covering lemma, see~\cite[Lemma 2.3]{B1} and~\cite[Lemma 1]{Anker}, to get a countable subset $\mathfrak{U}$ of $G$ such that $G=\bigcup_{x\in \mathfrak{U}} B(x,r)$ and such that for all $x_0\in \mathfrak{U}$, $B(x_0,2r) \cap B(x,2r)$  is non-empty for at most $n$ elements $x \in \mathfrak{U}$. Since $B(x,r) \cap (G\setminus B(e,r+R))\neq \emptyset $ implies $x\notin B(e,R)$, by~\eqref{eq3} we get
\begin{align*}
	\int_{G\setminus B(e,r+R)} \abs{f}^2\,\dd \mu 
	& \leq   \sum_{\substack{ x\in \mathfrak{U}  \\ B(x,r) \cap (G\setminus B(e,r+R))\neq \emptyset  }} \int_{B(x,r)}\abs{f}^2\,\dd \mu \\
	& \leq 	\epsilon  \sum_{x\in \mathfrak{U} } \int_{B(x,2r)} ( \abs{\nabla_\Hc f}^2+ (1+V)\abs{f}^2  )\,\dd \mu \\
	&\leq 2 \eps n Q(f).
\end{align*}
	The conclusion then follows from Proposition~\ref{prop1}.
\end{proof}

As a corollary, we obtain the aforementioned generalization of~\cite[Theorem 2]{Simon} and~\cite[Theorem 4.2]{BC1} for polynomially thin potentials. 

\begin{corollary}\label{cormagrpol}
Assume that for every $M>0$ there is $\ell>0$ such that $\Omega_{M}= \{y\in G \colon \, V(y)\leq M\}$ satisfies 
\[
\int_{\Omega_M} \mi(\Omega_M \cap B(x,r))^{\ell} \,\dd\mu(x) <\infty 
\]
for every $r>0$. Then $\Hs_V$ has purely discrete spectrum.
\end{corollary}

\begin{proof}
By Proposition~\ref{prop:2}, it is enough to show that, if a measurable subset $\Omega$  of $G$ is such that, for some $\ell>0$,
\[
\int_{\Omega} \mu( \Omega \cap B(x,r))^{\ell} \,\dd\mu(x) <\infty 
\]
for all $r>0$, then there is $s>0$ such that $\mu(\Omega \cap B(x,s))\to 0$ when $x\to \infty$. We shall actually show that such condition holds \emph{for all} $s>0$.

Pick $s>0$ and assume, by contradiction, that there exist a sequence $(x_n)$ going to $ \infty$ in $G$, and $\epsilon>0$ such that $\mu(\Omega \cap B(x_n,s))\geq \eps$ for all $n\in \N$. Without loss of generality, we can assume that $d(x_n, x_m) \geq 2s$ for all $n\neq m$, so that the balls $B(x_n,s)$ are pairwise disjoint. It follows that, for all $n\in \N$ and $x\in  \Omega \cap B(x_n,s)$,
\begin{equation}\label{Omegageqeps}
\mu(\Omega \cap B(x,2s))\geq \mu(\Omega \cap B(x_n,s)) \geq \eps.
\end{equation}
Therefore
\begin{align*}
\int_{\Omega} \mu( \Omega \cap B(x,2s))^{\ell} \,\dd \mu(x) 
& \geq \sum_n \int_{\Omega \cap B(x_n, s)  } \mu( \Omega \cap B(x,2s))^{\ell} \,\dd\mu(x) \\
& \geq  \epsilon^\ell \sum_n \mu( \Omega \cap B(x_n, s) )  = \infty,
\end{align*}
where the last equality follows from~\eqref{Omegageqeps}. The proof is complete.
\end{proof}

\subsection{A prelude to polynomial potentials}
The last result of this section is a characterization of the discreteness of the spectrum of Schr\"odinger operators whose potential is of the form $V=f\circ p$, where $p\colon G\to \R^n$ is a function in a finite-dimensional left-invariant space of functions and $f\colon \R^n\to \R$ is a proper map. We recall that a space $W$ of functions on $G$ is said to be left-invariant if the left-translates of its elements still belong to $W$. This result will lay the ground for the next section, where $p$ will be a polynomial map.  The characterization is again provided in terms of the decay property at infinity of the measure of the  sublevel sets of $V$. We begin with some lemmas.

\begin{lemma}\label{lem:8bis}
	Let $W$ be a finite-dimensional left-invariant space of functions from $G$ to some $\R^n$. Then, every element of $W$ is real analytic.
\end{lemma}

\begin{proof}
Observe that the map $L\colon G \to  GL(W)$ given by $x \mapsto L_x $ is a group homomorphism. In addition, if we endow $W$ with the unique Hausdorff topology which is compatible with its vector space structure and $GL(W)$ with the topology of pointwise convergence, then $L$ is continuous. Therefore, $L$ is real analytic by~\cite[Theorem 1, Ch.\ 3, \S 8, No.\ 1]{BourbakiLie2}. Hence, for every $p\in W$ the map $x \mapsto  L_{x^{-1}} p(e)=p(x)$ is real analytic.
\end{proof}

\begin{lemma}\label{lem:6}
Let $W$ be a finite-dimensional left-invariant space of functions from $G$ to some $\R^n$. Let $\delta,r>0$ be given. 
	Then, there is  $C>0$ such that, for every $M>0$, $x\in G$, and  $p\in W$, if 
	\[	
	\mu(\{y\in G \colon |p(y)|\leq M\}\cap B(x,r) )\geq \delta
	\]
then $\norm{\mathbf{1}_{B(x,r)} p }_\infty \leq C M$. 
\end{lemma}

\begin{proof}
	We may reduce to proving the assertion for $x=e$ and $M=1$, up to replacing $p$ with $M p(x^{-1}\,\cdot\,)$. Then, assume by contradiction that there is a sequence $(p_k)$ of elements of $W$ such that 
	\[
	\mu(\{y\in G \colon |p_k(y)|\leq 1\}\cap B(e,r) )\geq \delta
	\]
	for every $k\in\N$, and 
\begin{equation}\label{eqinfinf}
	\lim\limits_{k\to \infty}\norm{\mathbf{1}_{B(e,r)} p_k }_\infty=\infty.
\end{equation}
	Observe that the map $p\mapsto\norm{\mathbf{1}_{B(e,r)} p }_\infty $ is a continuous semi-norm on $W$; it is a norm by Lemma~\ref{lem:8bis}.
	
	Since $W$ is finite dimensional, we may assume that the sequence $(p_k/\norm{\mathbf{1}_{B(e,r)}  p_k}_\infty)$ converges to some $p\in W$.  By continuity, $\norm{\mathbf{1}_{B(e,r)}  p}_\infty=1$. 
	
Now, set 
\[
F_k = \{ y\in G\colon |p_k(y)|\leq 1\}\cap B(e,r),\qquad  k\in\N,
\]
so that, by assumption, $\mu(F_k)\geq \delta$. If we define 
\[
F = \bigcap_{k\in \N} \bigcup_{h\geq k} F_h,
\]
then $F\subseteq B(e,r)$ and
\[
\mu(F)=\lim\limits_{k\to \infty} \mu\bigg(\bigcup_{h\geq k} F_h\bigg)\geq \delta.
\]
Because of~\eqref{eqinfinf} and since $|p_k(y)|\leq 1$ for all $y\in F_k$,  one gets 
\[
|p(y)| = \lim_{k\to \infty} \frac{|p_k(y)|}{\|\mathbf{1}_{B(e,r)}  p_k\|_\infty} =0 \qquad \forall\, y\in F,
\]
namely $p(F)= \{0\}$. Therefore, $p$ vanishes on a set of strictly positive measure;  since it is analytic by Lemma~\ref{lem:8bis}, we obtain $p=0$, while $\norm{\mathbf{1}_{B(e,r)} p}_\infty=1$: this is a contradiction.
\end{proof}

\begin{proposition}\label{prop:1}
Let $W$ be a finite-dimensional left-invariant space of functions from $G$ to some $\R^n$. Assume that $V=f\circ p$ for some $p\in W$  and some proper map $f\colon \R^n\to \R$, and suppose $r>0$. Then $\Hs_V$ has purely discrete spectrum if and only if, for every $M>0$,
	\[
	\lim_{x\to \infty} \mu\big( \{ y\in G \colon V(y) \leq M\} \cap B(x,r) \big)=0.
	\] 
\end{proposition}

\begin{proof}
	Sufficiency is a special case of Proposition~\ref{prop:2}. Conversely, assume that there are $M, \delta>0$, and a sequence $(x_k)$ of elements of $G$ such that $(x_k)\to \infty$ for $k\to \infty$ and 
	\[
	\mu(\{ y\in G \colon V(y) \leq M\} \cap B(x_k, r) )\geq \delta.
	\]
	Observe that we may assume that the balls $B(x_k,r)$ are pairwise disjoint; in addition, since $f$ is proper, there is a constant $M'>0$ such that 
	\[
	\{ y\in G \colon V(y) \leq M\} \subseteq \{ y\in G \colon |p(y)| \leq M'\}. 
	\]
By Lemma~\ref{lem:6} we get that $\norm{\mathbf{1}_{B(x_k,r)} p}_\infty\leq C M'$ for every $k\in \N$, so that $V$ is uniformly bounded on $\bigcup_{k\in \N} B(x_k,r)$. Now, if $u\in C^\infty_c(B(e,r))$ is non-zero, then the sequence $(u(x_k^{-1}\,\cdot\,))$ is bounded in the domain of $\Hs_V$, but has no convergent subsequence. Hence, $\Hs_V$ does not have purely discrete spectrum by~\cite[Theorem 11.3.13]{Ol}.
\end{proof}

As its proof shows, Proposition~\ref{prop:1} can be rephrased in an equivalent way which will be convenient to us in the following. For future reference, we state it as a remark.
\begin{remark}\label{remequivprop:1}
Under the assumption of Proposition~\ref{prop:1}, $\Hs_V$ does \emph{not} have purely discrete spectrum if and only if there is a sequence $(x_k)$ in $G$ such that $x_k\to \infty$ and such that $p$ is uniformly bounded on the sets $B(x_k,r)$, $k\in\N$.
\end{remark}

\section{Polynomial potentials}\label{Sec:pol}
In this section we consider the special case when the potential $V$ is a polynomial, and obtain a characterization of the discreteness of the spectrum of $\Hs_V$ in terms of the vanishing properties of the right-invariant derivatives of $V$. Though inspired by~\cite{MetafunePallara}, such result needs substantially different ideas and techniques.  In the final part of the section we discuss some  notable examples, namely harmonic oscillators on Heisenberg groups. 

 We begin by recalling two definitions of  polynomials on general connected Lie groups. Given a left-invariant vector $X\in \mathfrak{g}$, we denote by $X^R$ the unique right-invariant vector field such that $X^R_e = X_e$. We follow closely~\cite{AntonelliLeDonne}; see also~\cite{Leibman}. 

\begin{definition}\label{defpol}
Let $m,n\in\N$ be given. 
\begin{itemize}
\item[(i)] A function $p\colon G\to \R^n$ of class $C^{m}$ is said to be a $\gf$-polynomial of degree at most $m$ if $(X^R)^m p=0$ for every $X\in \gf$.
\item[(ii)] A function $p\colon G\to \R^n$ of class $C^{m+1}$ is said to be a Leibman polynomial of degree at most $m$  if  $X_1^R\cdots X_{m+1}^R p=0$ for every $X_1,\dots, X_{m+1}\in \gf$.
\end{itemize}
The space of $\gf$-polynomials of degree at most $m$ will be denoted by $\mathcal{P}_m$.
\end{definition}
Observe that even though~\cite{AntonelliLeDonne} considers left-invariant vector fields, its theory holds as well with the right-invariant ones by considering the opposite group of $G$ (i.e., by considering the group which switches the role of the first and second factor). Moreover,~\cite{AntonelliLeDonne} considers only scalar-valued polynomials, but its results can be applied in our setting componentwise. Our definition of Leibman polynomials is an equivalent formulation of that in~\cite{Leibman} given in~\cite[Proposition 5.1]{AntonelliLeDonne}.

We first prove a preliminary result about the structure of the space of $\gf$-polynomials and the properties of its elements. 

\begin{lemma}\label{lem:9}
For all $m\in\N$, $\mathcal{P}_m$ is a finite-dimensional space and is invariant  both under left and right translations.  Further, every element of $\mathcal{P}_m$ is real analytic.
\end{lemma}

\begin{proof}
	The fact that $\mathcal{P}_m$ is finite dimensional and that its elements are real analytic is proved in~\cite[Theorem 1.1]{AntonelliLeDonne}. In addition, it follows from the definition that $\mathcal{P}_m$ is right invariant. We are then left with proving that $\mathcal{P}_m$ is also left invariant. For $g\in G$, define $\mathrm{Ad}_g\colon \gf\to\gf$ as the differential of the inner automorphism $x\mapsto g x g^{-1}$ of $G$, and let $X\in \gf$. Then, since $\mathrm{Ad}_{g^{-1}}X$ is left invariant and $(X^R) L_g = L_g (\mathrm{Ad}_{g^{-1}}X)^R$, one gets
	\[
	(X^R)^m L_g p= L_g [(\mathrm{Ad}_{g^{-1}}X)^R]^m p=0,
	\]
	so that $L_g p\in \mathcal{P}_m$ by the arbitrariness of $X$.
\end{proof}

The following theorem is one of the main results of this section. We recall that an exponential group is a group whose exponential is a global diffeomorphism. Notable examples are simply connected nilpotent Lie groups.
\begin{theorem}\label{Thm:carpolinomio}
Assume that $G$ is an exponential group, and that $V=f\circ p$ for a $\gf$-polynomial $p\colon G \to \R^n$ and a proper map $f\colon \R^n\to \R$.  Then, $\Hs_V$ has purely discrete spectrum if and only if there is no non-zero $X\in \gf$ such that $X^R p=0$.
\end{theorem}

\begin{proof}
Let $m\in \N$ be such that $p\in \mathcal{P}_m$. Assume that there is no non-zero $X\in \gf$ such that $X^R p = 0$, and by contradiction assume also that $\Hs_V$ does not have purely discrete spectrum. Fix $r>0$. By Remark~\ref{remequivprop:1}, there is a sequence $(x_j)$ of elements of $G$ such that $x_j\to \infty$ for $j\to \infty$ and such that the sequence $ (\mathbf{1}_{B(x_j,r)} p)$ is uniformly bounded.
	
Observe that,  since $\exp$ is onto,  there are $t_j\geq 0$ and $X_j\in \gf$ such that $\abs{X_j}=1$ and $x_j=\exp(t_j X_j)$.  Since $x_j\to \infty$, one has $t_j\to +\infty$; we may then assume that $t_j\geq 1$ for every $j\in\N$. 
	 Take $r'>0$ and $\eps>0$ such that 
\begin{equation}\label{pallaprima}
	 \exp(Y) B(e,r')\subseteq B(e,r) \qquad \forall \, Y\in B_\gf(0,\eps),
\end{equation}
	 and let us prove  that for every $k\in\N$ there is a constant $C_k>0$ such that, for all $X\in \overline B_\gf(0,1)$ and $t\in \R$,
\begin{equation}\label{claim1}
	\norm{\mathbf{1}_{\exp(t X) B(e,r')} (X^R)^k p}_\infty\leq C_k \norm{\mathbf{1}_{\exp(t X) B(e,r)} p}_\infty.
\end{equation}
 Indeed, observe first that  since $p\in \Pc_m$, for $X\in \gf$ and $y\in G$ one has
	\[
	\bigg(\frac{d}{dt}\bigg)^m p(\exp(t X) y) = (X^R)^m p(\exp(t X) y )  = 0.
	\]
Hence the functions $t\mapsto p(\exp(t X) y)$, $t\in \R$, are polynomials on $\R$ of degree at most $m-1$, whose vector space we denote by $W$. Since derivatives are linear operators on the finite dimensional vector space $W$, they are continuous with respect to any norm on it; in particular, for every $k\in\N$ there is a constant $C_k>0$ such that
	\[
	\abs{q^{(k)}(0) }\leq C_k\norm{\mathbf{1}_{(-\eps,\eps)}q}_\infty
	\]
	for every $q\in W$. By the translation invariance of $W$, one gets
\begin{equation}\label{boundedder}
	\abs{q^{(k)}(t) }\leq C_k\norm{\mathbf{1}_{(t-\eps,t+\eps)}q}_\infty
\end{equation}
	for every $q\in W$ and $t\in \R$. Hence, by ~\eqref{boundedder} applied to the polynomial 
	\[
	t\mapsto p(\exp(t X) y),
	\]
and by~\eqref{pallaprima}, for $X\in \overline B_\gf(0,1)$, $y\in B(e,r')$ and $t\in \R$ we get
	\[
	|(X^R)^k p(\exp(t X) y )  |\leq C_k \norm{\mathbf{1}_{(t-\eps, t+\eps)} p(\exp(\,\cdot\, X) y ) }_\infty  \leq C_k \norm{\mathbf{1}_{\exp(t X) B(e,r)} p}_\infty,
	\]
whence our claim~\eqref{claim1}.
	
	Now observe that, by (finite) Taylor expansion,
\[
	[(X^R)^k p](\exp(t X) y)= \sum_{h=0}^{m-k-1} \frac{t^h}{h!} (X^R)^{k+h}p(y)
\]
	for every $X\in \gf$, $t\in \R$,  $y\in G$, and $k=0,\dots, m-1$.  By~\eqref{claim1}, for $k=0,\dots,m-1$ the functions
	\[
[(X_j^R)^k p](\exp(t_j X_j)  \, \cdot \, ) =\sum_{h=0}^{m-k-1} \frac{t_j^h}{h!} (X_j^R)^{k+h}p=\sum_{\ell=0}^{m-1} a^{(j)}_{k,\ell} t_j^{\ell} (X^R_j)^\ell p
	\]
	are uniformly bounded on $B(e,r')$ for $j\in \N$, where we set
\[
 a_{k,\ell}^{(j)}=
\begin{cases} 
 \frac{1}{(\ell-k)! t_j^k} \qquad & \mbox{if } 0\leq k\leq \ell \leq m-1\\
 0  & \mbox{if } 0\leq k,\ell\leq m-1, \; \ell<k.
\end{cases}
\]
Since, for every $j$, the matrices $(a_{k,\ell}^{(j)})$ are upper triangular with  diagonal elements constantly equal to $1$, for $j\in \N$ they  are uniformly bounded and have determinant $1$. Hence, the matrices $(a_{k,\ell}^{(j)})^{-1}$ are uniformly bounded as well, as $j\in \N$. Therefore, we deduce that the functions
	\[
	t_j^{\ell} (X_j^{R})^\ell p, \quad j\in \N,
	\]
	are uniformly bounded on $B(e,r')$ for every $\ell=0,\dots,m-1$. In particular, the functions
	\[
	t_j X_j^{R} p, \quad j\in \N,
	\]
	are uniformly bounded on $B(e,r')$. 
	
Since $\abs{X_j}=1$ for all $j$, up to considering a subsequence we may assume that $(X_j)$ converges to some $X_\infty$ in $\gf$, with $\abs{X_\infty}=1$.	Moreover, since $t_j \to \infty$ and $(t_j X_j^{R} p)$ is bounded on $B(e,r')$, the sequence $ (X_j^{R} p)_j$ converges uniformly to $0$ on $B(e,r')$, so that $X_\infty^R p=0$ on $B(e,r')$. Since $p$ is analytic  by Lemma~\ref{lem:9}, $X_\infty^R p=0$ on $G$: contradiction.
	
\smallskip	
	
	Conversely, assume that $X^{R}p=0$ for some non-zero $X\in \gf$, and let us prove that $\Hs_V$ does not have purely discrete spectrum. By Taylor expansion, as before,
	\[
	p(\exp(t X)y)=\sum_{h=0}^{m-1} \frac{t^h}{h!} (X^R)^h p (y) = p(y),
	\]
	whence the left hand side is uniformly bounded as $t\in \R$ and $y$ lies in a relatively compact open subset of $G$. Since $\exp(tX) \to \infty$ when $t\to \infty$, the conclusion follows from Remark~\ref{remequivprop:1}.
\end{proof}

The next goal in this section is to prove an analogue of Theorem~\ref{Thm:carpolinomio} on groups which are not exponential. To do this, we shall consider Leibman polynomials, and reduce the problem from a non-exponential group to a nilpotent group. The reason for considering Leibman polynomials here is that they behave well when passing to a quotient with respect to an element of the lower central series of the group;  and this will be one of the steps of our reduction.

The characterization that we obtain is in terms of vector fields whose flow is proper. In this order of ideas, we proceed with proving Lemmas~\ref{lemma:quotient} and~\ref{lem:10} below. Since the latter will be applied to some groups other than $G$, we distinguish from $G$ the Lie group therein.

\begin{lemma}\label{lemma:quotient}
Let $p\colon G\to \R^n$ be a Leibman polynomial. Then,  there is a closed connected normal subgroup $N$ of $G$ such that $G/N$ is simply connected and nilpotent, and such that $p = q\circ \pi$, where $\pi \colon G\to G/N$ is the canonical projection and $q$ is a Leibman polynomial on $G/N$.
\end{lemma}

\begin{proof}
By the definition of a Leibman polynomial, there is an element $\widetilde{\mathfrak{h}}$ of the lower central series of $\gf$ such that 
\begin{equation}\label{htilde}
X^R p=0, \qquad \forall X\in\widetilde{\mathfrak{h}}.
\end{equation}
For the sake of clarity, we summarize here the strategy of the proof.
\begin{itemize}
\item[Step 1.] We show that if $H$ is the closure of the integral subgroup of $G$ corresponding to $\widetilde{\hf}$, then $p$ is right $H$-invariant, hence it induces a Leibman polynomial on the nilpotent Lie group $G/H$, whose Lie algebra we denote with $\mathfrak{m}$.
\item[Step 2.] We observe that $G/H$ is the quotient of the simply connected Lie group $M$ with same Lie algebra $\mathfrak{m}$, modulo a discrete subgroup $D$ of the center of $M$.
\item[Step 3.] We prove that the polynomial on $M$ induced by $p$ is right invariant with respect to the smallest integral subgroup of $M$ containing $D$, so that it induces a Leibman polynomial on the quotient. We then reconstruct $p$ in terms of this last Leibman polynomial.
\end{itemize}

{Step 1.} Let $\widetilde H$ be the connected integral subgroup of $G$ whose tangent Lie algebra is $\widetilde \hf$, and observe that $\widetilde H$ is a normal subgroup, cf.~\cite[Def.\ 1, Thm.\ 2, Prop.\ 14, Ch.\ III, \S 6]{BourbakiLie2}. Then, the closure $H$ of $\widetilde H$  is a closed normal Lie subgroup of $G$, hence its Lie algebra $\hf$ is an ideal of $\gf$ containing $\widetilde \hf$.

We now show that 
\begin{equation}\label{Hinvariance}
p(x y)=p(x) \qquad \forall \,x\in G, \; y\in H,
\end{equation}
so that $X^R p=0$ for every $X\in \hf$. By~\eqref{htilde}, property~\eqref{Hinvariance} holds when $y\in \exp_G(\widetilde \hf)$, hence when $y$ is a finite products of   elements of $\exp_G(\widetilde \hf)$. Since $\exp_G(\widetilde \hf)$ is a symmetric neighbourhood of the origin in $\widetilde H$,~\eqref{Hinvariance} holds for every $y\in \widetilde  H$. By continuity, this property extends to every $y\in H$.

By~\eqref{Hinvariance}, $p$ is constant on $xH$ for all $x$, hence it induces a function $p_1$ on $G/H$. In other words, if $\pi_1\colon G\to G/H$ is the canonical projection, then
	\[
	p=p_1\circ \pi_1.
	\]
	 By~\cite[Proposition 1.11]{Leibman}, $p_1$ is a Leibman polynomial on $G/H$.

{Step 2.} Let $M$ be a simply connected Lie group with nilpotent Lie algebra $\mathfrak{m}= \gf/\hf$, and observe that, since $M$ is nilpotent, it may be identified with its Lie algebra by means of the exponential map. In addition, by~\cite[Theorem 3, Ch.\ III, \S6, No.\ 3]{BourbakiLie2}, there is an  analytic homomorphism $\pi_2\colon M\to G/H$ such that $\pi_2$ is onto and $\ker \pi_2$ is a discrete subgroup $D$ of the centre $Z$ of $M$. Observe that $Z$ can be identified with the centre $\mathfrak{z}$ of $\mathfrak{m}$ (as a manifold and as a group) by means of the exponential map by~\cite[Propositions 13 and 15, Ch.\ III, \S 9, No.\ 5]{BourbakiLie2}. We define $D'= \exp_{M}^{-1}(D)$, and observe that $D'$ is a (closed) discrete subgroup of $\mathfrak{z}$.

Step 3. Again by~\cite[Proposition 1.11]{Leibman}, $p_1\circ \pi_2$ is a Leibman polynomial on $M$; that is, since $M$ is nilpotent,
\[
p_2\coloneqq p_1\circ \pi_2\circ \exp_{M}
\]
is a polynomial on $\mathfrak{m}$, by~\cite[Corollary 1.4]{AntonelliLeDonne}. 

In addition, since $D=\ker \pi_2$, one has $p_2(X_1+X_2)=p_2(X_1)$ for every $X_1\in \mathfrak{m}$ and $X_2\in D'$, hence for every $X_2$ in the vector subspace $\mathcal{V}$ of $\mathfrak{z}$ generated by $D'$. Therefore, $p_2$ induces a polynomial map $p_3$ on the simply connected nilpotent Lie group $M/\exp_{M} \mathcal{V}$, i.e.\
\[
p_2=p_3\circ \pi_\mathcal{V} \circ \exp_M
\]
	where $\pi_\mathcal{V} \colon M\to M/\exp_M \mathcal{V}$ is the canonical projection.
	
To conclude, observe that $\exp_{M}\mathcal{V}/D$ is a closed connected normal subgroup of $M/D$, so that there are a closed connected normal subgroup $N$ of $G$ (containing $H$), an isomorphism $\pi_3\colon M/\exp_{M} \mathcal{V}\to G/N$, and a Leibman polynomial $q$ on $G/N$ such that 
\[
q\circ \pi_3=p_3.
\]
Therefore,  $q\circ \pi=p$, where $\pi\colon G\to G/N$ denotes the canonical projection.
\end{proof}

\begin{lemma}\label{lem:10}
Let $H$ be a connected noncompact Lie group with Lie algebra $\mathfrak{h}$. Then, there is $X\in \hf$ such that the map $t \mapsto \exp_{H}(t X)$ is proper.
\end{lemma}

\begin{proof}
We claim that the statement holds when $\hf$ is compact  and when $\hf$ is solvable. Assuming the claim, we complete the proof.

By~\cite[Theorem 3]{Wustner}, $H$ can be decomposed as $H=K S$ where $K$ is a closed connected subgroup of $H$ whose Lie algebra  is compact, and $S$ is a solvable connected subgroup of $H$. Observe that $\overline S$ is a solvable connected subgroup of $H$ by~\cite[Corollary 2 to Proposition 1, Ch.\ III, \S 9, No.\ 1]{BourbakiLie2}, so that we may assume that $S$ is closed in $H$. Since $H$ is not compact, either $K$ or $S$ is not compact. 

Therefore $H$ has a closed, connected and noncompact subgroup $H_1$, with Lie algebra $\mathfrak{h}_1$, with the property that, by the claim, there is $Y\in\mathfrak{h}_1$ such that $t \mapsto \exp_{H_1}(t Y)$ is proper. By identifying $\mathfrak{h}_1$ with a subalgebra of $\hf$, we get $\exp_{H_1}=\exp_{H}$ on $\hf_1$.  Since the embedding $H_1 \to H$ is proper, as $H_1$ is closed in $H$, the conclusion follows. We are then left with proving the claim.

\smallskip

Assume first that $\hf$ is compact. By~\cite[Proposition 5, Ch.\ IX, \S 1, No.\ 4]{BourbakiLie5},  there are a closed, central, simply connected subgroup $N$ of $H$ and a connected compact subgroup $K$ of $H$ such that $H$ is the direct product $N\times K$. Observe that $N$ is an abelian group and, since $H$ is not compact, $N$ is not trivial. If $\mathfrak{n}$ is the Lie algebra of $N$, identified with an ideal of $\mathfrak{h}$, then $\exp_{H}$ induces an isomorphism of $\mathfrak{n}$ onto $N$, so that, for every non-zero $X\in \mathfrak{n}$, the map $ t \mapsto \exp_{H}(t X) $ is proper.
	
\smallskip	
	
Assume now that $\hf$ is solvable. Then, there are a simply connected Lie group $S$ with Lie algebra $\hf$, and a surjective analytic homomorphism $\pi\colon S\to H$ such that $\dd\pi$ is the identity and $\ker \pi$ is a discrete closed subgroup of the centre of $S$ (cf.~\cite[Theorem 3, Ch.\ III, \S 6, No.\ 3]{BourbakiLie2}).

\noindent By~\cite[Theorem 1]{Chevalley}, there are a basis $(X_1,\dots, X_k)$ of $\hf$ as a vector space and integers $0\leq r\leq k$, $1\leq j_1<\cdots<j_r\leq k$ such that 
\begin{itemize}
\item[(i)] $\ker \pi$ is a free abelian group of rank $r$;
\item[(ii)] the map 
	\[
	\varphi\colon \R^k \to S, \qquad  \varphi(t) = \exp_{S}(t_1 X_1)\cdots \exp_{S}(t_k X_k),
	\]
	is an analytic bijection;
\item[(iii)] $X_{j_1},\dots, X_{j_r}$ generate an abelian subalgebra of $\hf$;
\item[(iv)] $( \exp_{S}(X_{j_1}),\dots,\exp_{S}(X_{j_r}) )$ is a basis of $\ker \pi$ as a $\Z$-module.
\end{itemize}
Observe that if $r=k$, then $\hf$ is abelian, hence $S=\R^k$ and, up to a change of coordinates, $\ker \pi=\Z^k$. Thus $H\cong \T^k$, in particular $H$ is compact. Since this is not the case, we actually have $r<k$. 
Then, there is $j_0\in \{1,\dots, k\}\setminus \{j_1,\dots, j_r\}$. Let us prove that the map $t \mapsto \exp_{H}(t X_{j_0})$ is proper.
	
Assume by contradiction that there exists a compact subset $L$ of $H$ such that $A\coloneqq \{t\in \R\colon \exp_{H}(t X_{j_0}) \in L\}$ is unbounded. Notice that since $\exp_H(n X_{j_\ell} )=\exp_H(X_{j_\ell})^n$ is in the centre of $S$ for every $\ell=1,\dots, r$ and $n\in\Z$, 
\[
\varphi(t) \prod_{\ell=1}^r \exp_H( X_{j_\ell})^{n_{j_\ell}}= \varphi\bigg(t+ \sum_{\ell=1}^r n_\ell e_{j_\ell}\bigg)
\]
for every $t\in \R^k$ and $n_1,\dots, n_r\in \Z$, where $(e_j)_{j=1,\dots,k}$ is the canonical basis of $\R^k$. In addition, by~\cite[Proposition 18, Ch.\ IX, \S 2, No.\ 10]{BourbakiGT2}, there is a compact subset $L'$ of $\R^k$ such that $L=\pi(\varphi(L'))$, so that
	\[
A e_{j_0}\subseteq	\varphi^{-1}(\pi^{-1}(L))= \sum_{\ell=1}^r \Z e_{j_\ell}+ L'.
	\]
This is a contradiction because $A e_{j_0}$ is unbounded, while $\R e_{j_0}\cap (\sum_{\ell=1}^r \Z e_{j_\ell}+L')$ is bounded.
\end{proof}

\begin{proposition}\label{prop:carpolinomio}
Assume that $V=f\circ p$ for some Leibman polynomial $p\colon G \to \R^n$ and a proper map $f\colon \R^n \to \R$. Then, $\Hs_V$ has purely discrete spectrum if and only if there is no $X\in\gf$ such that $X^R p=0$ and such that the map $ t\mapsto \exp_G(t X)$ is proper.
\end{proposition}

\begin{proof}
By Lemma~\ref{lemma:quotient}, there is a closed connected normal subgroup $N$ of $G$ such that $G/N$ is simply connected and nilpotent, and such that $p$ factors through the canonical projection $\pi \colon G\to  G/N$, namely $p=q \circ \pi$, where $q$ is a Leibman  polynomial on $G/N$.
	
Fix $r>0$. By Remark~\ref{remequivprop:1}, $\Hs_V$ does not have purely discrete spectrum if and only if there is a sequence $(x_k)$ in $G$ such that $x_k\to \infty$ and such that $p$ is uniformly bounded on the sets $B(x_k,r)$, $k\in\N$. This happens if and only if $q$ is uniformly bounded on the sets $\pi(B(x_k,r))$, $k\in\N$. 
Two possibilities may arise:
\begin{enumerate}
\item[(a)] the sequence $(\pi(x_k))$ stays in a compact subset $K$ of $G/N$;
\item[(b)] there is a subsequence of $(\pi(x_k))$ which goes to $\infty$ in $G/N$.
\end{enumerate}

If (a) holds, then $\ker \pi =N$ is not compact. Therefore, there is $X$ in the ideal of $\gf$ tangent to $N$ such that the map $t\mapsto \exp_G(t X)\in G$ is proper by Lemma~\ref{lem:10}, and clearly $X^R p=X^R(q\circ \pi)=0$. Thus, the ``only if'' part  is proved when condition (a) holds. 

Assume instead that (b) holds. By Remark~\ref{remequivprop:1} applied to $G/N$, $\dd\pi(\Ls)$, and $f\circ q$, the operator $\dd \pi(\Ls)+f\circ q$ does not have purely discrete spectrum. Then, by Theorem~\ref{Thm:carpolinomio}, there is a non-zero right-invariant vector field $Y$ on $G/N$ such that $Y q=0$ on $G/N$.  Since there is  $X\in \gf$  such that $\dd \pi(X)^R= Y $ as $\dd\pi$ is surjective, we deduce $X^R p=0$. In addition, the map $t\mapsto \exp_G(t X)\in G$ is proper, for otherwise there would be a compact subset $L$ of $G$ such that $\exp_G(A X)\subseteq L$ for some unbounded subset $A$ of $\R$, whence $\exp_{G/N}(A Y^R )\subseteq \pi(L)$, which is a contradiction. Thus, the ``only if'' part is proved also if condition (b) is satisfied,  hence its proof is complete.
	
To show the ``if'' part,  by Remark~\ref{remequivprop:1} it suffices to show that, if there is $X\in\gf$ such that $X^R p=0$ and such that the map $ t\mapsto \exp_G(t X) $ is proper, then the sequence $(\exp_G(k X))$ in $G$  converges to $\infty$ and   $p$ is uniformly bounded on  the balls $B(\exp_G(k X),r)$, $k\in\N$. Since $p$ is then constant on the integral curves of $X^R$, which are of the form $t\mapsto  \exp_G(t X)x$ for $x\in G$, the assertion follows.
\end{proof}

We conclude this section with some remarks and notable examples of Schr\"odinger operators with polynomial potentials on the Heisenberg groups. 

\subsection{Examples and remarks.~The Heisenberg groups}
For $n\in \N$, the $(2n+1)$-dimensional  Heisenberg group $\mathbb{H}^n$ is the $2$-step stratified Lie group whose underlying manifold is $ \R^{n}\times \R^n\times \R$, endowed with the group law
\[
(x, y, t) \cdot (x', y', t')  =  \Big(x + x', y + y', t + t' + \frac{1}{2}(x\cdot  y' - y\cdot x')\Big), \quad x,y,x',y'\in \R^n, \;\; t,t'\in \R.
\]
Its Lie algebra $\mathfrak{h}^n$ is spanned, as a vector space, by the left-invariant vector fields $ X_1,\dots, X_n,Y_1, \dots, Y_n,T$ given by
\[
X_j= \partial_{x_j} - \frac{1}{2}y_j\partial_t, \qquad Y_j= \partial_{y_j} + \frac{1}{2}x_j\partial_t, \qquad j=1,\dots, n,\qquad  T= \partial_t, 
\]
which satisfy the relations
\[
[X_j,Y_j] = T, \quad [X_j,T] = [Y_j,T]=0, \qquad j=1,\dots, n.
\]
Their corresponding right-invariant vector fields are 
\[
X_j^R= \partial_{x_j} + \frac{1}{2}y_j\partial_t, \qquad Y_j^R= \partial_{y_j} - \frac{1}{2}x_j\partial_t, \qquad j=1,\dots, n,\qquad  T^R= \partial_t.
\]
The group $\mathbb{H}^n$ is unimodular, and if $\mu$ is the left (and right) Haar measure, namely the Lebesgue measure, then $\Ls = - \sum_{j=1}^n (X_j^2+Y_j^2)$ is a sum-of-squares sub-Laplacian. If $n=1$, we just write $X_1=X$ and $Y_1=Y$.

 It has recently become of interest  the problem of introducing on $\mathbb{H}^n$ an analogue of the Euclidean harmonic oscillator  $\Delta + |\cdot|^2$, where $\Delta = - (\partial_{x_1}^2 + \dots + \partial_{x_n}^2)$ is the classical Laplacian. Since the Euclidean harmonic oscillator is a Schr\"odinger operator (with $|\cdot|^2$ as a potential), it is reasonable to expect that its analogue on $\mathbb{H}^n$ is a Schr\"odinger operator as well. Two different candidates have been proposed, in~\cite{Fischer} and~\cite{RR1, RR2} respectively, as Schr\"odinger operators in the form $\Ls + V$, for some polynomial potentials $V$, and a quantitative description of their spectrum has been given. By Theorem~\ref{Thm:carpolinomio}, we can obtain at once a qualitative, though not quantitative, description of their spectrum and shed some light on the nature of their difference. 
 
In~\cite{Fischer}, it was proved that, if $n=1$, then the spectrum of
\[
-X^2 - Y^2 + x^2 + y^2
\]
is a half-line. In particular, it is \emph{not} discrete; this latter fact can be easily obtained from Theorem~\ref{Thm:carpolinomio}, as $\partial_t (x^2+y^2) =0$. In~\cite{RR1, RR2}, an harmonic oscillator was defined by means of the representation theory of the Dynin--Folland group, having the form
\[
- \sum_{j=1}^n (X_j^2+Y_j^2) + \gamma t^2,
\]
for a suitable $\gamma>0$.  It was shown in~\cite{RR2} that its spectrum is discrete, and asymptotic estimates were given for the distribution of its eigenvalues.  Again, the discreteness of the spectrum of this operator follows  at once from Theorem~\ref{Thm:carpolinomio}; the reader can in fact easily verify that there are no nonzero right-invariant vector fields which annihilate $(x,y,t)\mapsto \gamma t^2$.

Observe, moreover, that there are several polynomial potentials (in the wide sense as above, possibly involving a proper map) which may give rise to good candidates for harmonic oscillators on $\mathbb{H}^n$. For example, one may consider a Schr\"odinger operator $\Hs_V$ where $V =N^2$, being $N$ a homogeneous norm on $\mathbb{H}^n$ such as the Kaplan norm
\[
N(x,y,t)  = \sqrt[4]{( |x|^2 + |y|^2)^2 + 16t^2}.
\]
By Proposition~\ref{prop:2},  such an $\Hs_V$ has purely discrete spectrum.

\smallskip

As a final remark, we also observe that one cannot replace the right-invariant vector fields with the left-invariant ones in Theorem~\ref{Thm:carpolinomio}, in general. The polynomial 
\[
p(x,y,t) = y^2 x +2 y t
\] on $\mathbb{H}^1$ is annihilated by $X$, while it is not annihilated by any right-invariant vector fields, hence by Theorem~\ref{Thm:carpolinomio} the corresponding Schr\"odinger operator has purely discrete spectrum. Analogously, the polynomial \[
p(x,y,t) = y^2 x -2 y t
\]
on $\mathbb{H}^1$ is annihilated by $X^R$, while it is not annihilated by any left-invariant vector field.

\section{Muckenhoupt potentials}\label{Sec:Muck}

The aim of this section is to study the case when the potential $V$ is a local Muckenhoupt weight. To do this, we first develop a basic theory  for such weights on $G$; recall that as a metric measure space, $G$ is locally doubling but, in general, not doubling. For more information about Muckenhoupt weights on \emph{doubling} measure spaces, we refer the reader to~\cite{StrombergTorchinsky}.

For $R>0$, let $\mathcal{B}_R$ be the set of all balls of radius at most $R$. If $B$ is a ball of radius $r$ and $\lambda>0$, we denote by $\lambda B$ the ball with same centre and radius $\lambda r$. For a non-zero positive $w\in L^1_\loc$, we denote by $\mi_w$ the measure with density $w$ with respect to $\mi$, and consider the following conditions.
	\begin{enumerate}
		\item[\textnormal{(1)}] There are $\eps,\delta\in (0,1)$ such that, for every ball $B\in \mathcal{B}_R$ and every Borel subset $F$ of $B$,
		\[
		\mi(F)\leq \eps \mi(B) \implies \mi_w(F)\leq \delta \mi_w(B).
		\]
		\item[\textnormal{(2)}] There are $p\in [1,\infty)$ and $C>0$ such that, for every ball $B\in \mathcal{B}_R$ and every Borel subset $F$ of $B$,
		\[
		\frac{\mi_w(F)}{\mi_w(B)}\geq C \left( \frac{\mi(F)}{\mi(B)} \right)^p.
		\]
		\item[\textnormal{(3)}] There are $p\in (1,\infty)$ and $C>0$ such that, for every ball $B\in \mathcal{B}_R$,
		\[
		\bigg(\dashint_{B} w\,\dd \mi \bigg)\bigg( \dashint_{B} w^{-p'/p}  \bigg)^{p/p'}\leq C.
		\]	
		\item[\textnormal{(4)}] There are $\delta,c>0$ such that, for every ball $B\in \mathcal{B}_R$,
		\begin{equation*}
			\mu\bigg( \bigg\{ x\in B \colon w(x)\geq \delta \dashint_B w\, \dd \mu\bigg\} \bigg) \geq c \mu(B).
		\end{equation*}
	\end{enumerate}
Note that, a priori, condition (2) is meaningless if $\mi_w(B)=0$. However, if (2) holds for all $B\in \Bc_R$ such that $\mi_w(B)>0$, then the local doubling property of $\mi$  (recall~\eqref{localdoubling}) implies that $\mi_w$ is doubling on the balls in $\Bc_R$. Since $G$ is connected and $\mi_w\neq0$ by assumption, one gets $\mi_w(B)>0$ for all balls $B$.

 The following result is a local version of~\cite[Corollary 14]{StrombergTorchinsky}, and will lead us to the definition of some classes of local Muckenhoupt weights.

\begin{proposition}\label{prop:ST}
Let $R>0$ and a positive $w\in L^1_\loc$ be given.	Then
\[
(4) \implies (1), \qquad (3) \implies (2) \implies (1), \qquad\mbox{ and}  \qquad \ (3) \implies (4).
\]
If there is a constant $C>0$ such that
\begin{equation}\label{5r}
	\mi_w(5B)\leq C\mi_w(B)
\end{equation}
	for every $B\in \mathcal{B}_R$, then conditions \textnormal{(1)--(4)} are equivalent.
\end{proposition}

\begin{proof}
	The arguments used to prove~\cite[Corollary 14]{StrombergTorchinsky} show that
	\textnormal{(3)}$\implies$\textnormal{(2)}$\implies$\textnormal{(1)}, and that if	there is a constant $C>0$ such that~\eqref{5r} holds	for all $B\in \mathcal{B}_R$, then (1) implies (3).
	
	We now prove that (4) implies (1). By (4), there are $\delta,c>0$ such that 
\begin{equation}\label{widetildexr}
	\mu(E_\delta(B))\geq c\mi(B)
\end{equation}
	for all $B\in \mathcal{B}_R$, where to simplify the notation we have set
\begin{equation}\label{widetildedef}
 E_\delta(B) = \bigg\{y\in B\colon w(y)\geq \delta \dashint_{B} w\,\dd \mu\bigg\}.
\end{equation}
Notice that $c\leq 1$.
	
	Set $\alpha = c/2$ and $\beta = 1-c\delta/2$. Take $B\in \mathcal{B}_R$, and let $F$ be a Borel subset of $B$ such that $\mu(F)\leq \alpha\mi (B)$. Then
	\begin{align*}
		\mu(E_\delta(B)\setminus F) & = \mu(B \setminus F) + \mu(E_\delta(B)) - \mu((B\setminus F) \cup  E_\delta(B))\\
		& \geq (1-\alpha )\mu(B) + c \mu(B )  - \mu(B) \\
		&= \frac{c}{2}\mu(B ),
	\end{align*}
	so that by~\eqref{widetildexr}
	\[
	\int_{B\setminus F} w\,\dd\mi\geq \int_{ E_\delta(B)\setminus F} w\,\dd \mu\geq \frac c 2 \mu(B) \delta \dashint_{B}w \,\dd \mu= (1-\beta) \int_{B} w\,\dd \mu.
	\]
	Hence,
	\[
	\mi_w(F)\leq \beta \mi_w(B),
	\]
	which proves (1). 
	
To conclude the proof, it will then be enough to prove that (3) implies (4). More precisely, we shall prove that for every $c\in (0,1)$ we may find $\delta\in (0,1)$ so that (4) holds. For $B\in \mathcal{B}_R$, set 
	\[
	\overline w = \inf\{\; t>0\colon \mi(\{ y\in B\colon w(y)>t  \})\leq c\mi(B)\; \},
	\]
	so that clearly
	\[
	\mi(\Set{y\in B \colon w(y)>\overline w})\leq c\mi(B),
	\]
	that is,
	\begin{equation}\label{wbarrato}
		\mi(\Set{y\in B \colon w(y)\leq\overline w})\geq (1-c)\mi(B).
	\end{equation}
	Observe that, in addition,
	\begin{equation}\label{wbarratobis}
		\mi(\Set{y\in B \colon w(y)\geq\overline w})\geq c\mi(B).
	\end{equation}
	By~\eqref{wbarrato}
	\[
	\bigg(\dashint_{B} w^{-p'/p}\,\dd \mi \bigg) ^{p/p'}\geq  (1-c)^{p/p'}\overline w^{-1},
	\]
	and this together with (3) implies
	\[
	\dashint_{B} w\,\dd \mi\leq C \bigg(\dashint_{B} w^{-p'/p}\,\dd \mi \bigg) ^{-p/p'}\leq (1-c)^{-p/p'}C\overline w.
	\]
	Therefore,
	\[
	\Set{y\in B\colon w(y)\geq\overline w }\subseteq E_{(1-c)^{p/p'}C^{-1}}(B),
	\]
	whence, by~\eqref{wbarratobis},
	\[
	\mi(E_{(1-c)^{p/p'}C^{-1}}(B))\geq c\mi(B).
	\]
	The proof is complete.
\end{proof}

The previous result brings us to the following definitions. For notational convenience, we shall maintain the notation~\eqref{widetildedef}.
\begin{definition}
We define $A_{\infty, \loc}$ as the space of non-zero positive functions $w\in L^1_\loc$ for which there are $R,\delta,c>0$ such that
\begin{equation}\label{AinftyR}
 \mu (E_\delta(B)) \geq c \mu(B)
\end{equation}
for all $B\in \Bc_R$. We define $\widetilde A_{\infty, \loc}$ as the space of $w\in A_{\infty,\loc}$ such that there are $R,C>0$ such that
\[
\mi_w(2B)\leq C\mi_w(B)
\]
for every ball $B\in \mathcal{B}_R$. If $p\in [1,\infty)$, we define $A_{p,\loc}$ as the set of non-zero positive functions $w\in L^1_\loc $ for which there are $R,C>0$ such that
	\[
	\bigg(\dashint_{B} w\,\dd \mi \bigg)\bigg( \dashint_{B} w^{-p'/p}  \bigg)^{p/p'}\leq C
	\]
	if $p>1$, and
	\[
	\bigg(\dashint_{B} w\,\dd \mi\bigg)\norm{ w^{-1}}_{L^\infty(B)}\leq C
	\]
	if $p=1$,	for every $B\in \Bc_R$.
\end{definition}

By Jensen's inequality, $A_{p,\loc}\subseteq A_{q,\loc}$ for every $1\leq p\leq q<\infty$, and  by the implication $(3)\implies (4)$ in Proposition~\eqref{prop:ST}, $A_{p,\loc} \subseteq  A_{\infty,\loc}$ for all $p\in (1,\infty)$. We are now going to show that the spaces $A_{p,\loc}$, $p>1$, and $\widetilde A_{\infty,\loc}$ may be equivalently defined by requiring that the stated conditions hold for \emph{every} $R>0$ (instead for \emph{some} $R>0$).

\begin{proposition}\label{prop:5}
If  $w\in \widetilde A_{\infty,\loc}$, then the measure $\mi_w$ is locally doubling and conditions \emph{(1)--(4)} hold for every $R>0$.
\end{proposition}

\begin{proof}
Define $\widehat A_{p,R}$, for every $R>0$ and $p\in (1,\infty)$,  as the set of positive non-zero  functions $w\in L^1_\loc(\mi)$ for which there is a constant $C>0$ such that
\begin{equation}\label{cond2late}
	\frac{\mi_w(F)}{\mi_w(B)}\geq C \left(\frac{\mi(F)}{\mi(B)}\right)^{p}
\end{equation}
	for every ball $B\in \mathcal{B}_R$, and every Borel subset $F$ of $B$.
	
	We claim that 
\begin{equation}\label{claimwidehatA}
	\widehat A_{p, R}=\widehat A_{p,R'} \qquad \forall \, R,R'>0.
\end{equation}
The statement follows from the claim, as we now show. If $w\in \widetilde A_{\infty,\loc}$,  then there are $p\in (1,\infty)$ and $R>0$ such that $w\in \widehat A_{p,R}$, by Proposition~\ref{prop:ST}. Hence, $w\in \widehat A_{p,R'}$ for every $R'>0$, by the claim. In particular, the measure $ \mi_w$  is locally doubling (apply~\eqref{cond2late} with $F=\frac 1 2 B$, and use the fact that $\mi$ is locally doubling). Therefore, the assertion follows from  Proposition~\ref{prop:ST}.
	
It then remains to prove the claim~\eqref{claimwidehatA}. Pick $R>0$ and $w\in \widehat A_{p, R}$, so that there is a constant $C_1>0$ such that~
\begin{equation}\label{cond2xr}
	\frac{\mi_w(F)}{\mi_w(B(x,r))}\geq C_1 \left(\frac{\mi(F)}{\mi(B(x,r))}\right)^{p}
\end{equation}
	for every $x\in G$ and $r\in (0,R]$, and for every Borel subset $F$ of $B(x,r)$.
	Observe that, since $\mi$ is locally doubling, there is a constant $C_2>0$ such that
\begin{equation}\label{C2}
	\mi(B(e,2r))\leq C_2 \mi(B(e,r))
\end{equation}
	for every $r\in (0, (5/4) R]$.
	Combining the preceding inequalities, we then see that 
	\[
	\mi_w(B(x,2r))\leq C_3 \mi_w(B(x,r))
	\]
	for every $x\in G$ and $r\in (0,R/2]$, where $C_3\coloneqq C_1^{-1} C_2^{p}$.
	
	Then, fix $x\in G$, $r\in (R, (5/4)R]$, and a Borel subset $F$ of $B(x,r)$.	Let $(y_j)_{j\in J}$ be a (finite) family of points in $B(x, (3/4) R)$ which is maximal for the property that $d(y_j, y_{k})\geq R/2$  for every $j,k\in J$ with $j\neq k$.
	Then, the balls $B(y_j, R/4)$ are pairwise disjoint and contained in $B(x,R)$; let us prove that 
\[
\bigcup_{j\in J}B(y_j,R) \supseteq B(x, (5/4)R) \supseteq B(x,r).
\]	
	Indeed, take $y\in B(x, (5/4) R)$, and let $\gamma$ be a minimizing geodesic joining $x$ and $y$, which exists by~\cite[Corollaries 3.49 and 7.51]{AgrachevBoscainBarilari}. Then, there is $z$ in the support of $\gamma$ such that $d(x,z)<(3/4) R$ and $d(z,y)<R/2$. By maximality, there is $j\in J$ such that $d(z,y_j)<R/2$, so that $d(y,y_j)\leq d(y,z)+d(z,y_j)<R$, whence $y\in B(y_j, R)$.
	
	 In addition, since
\[
\mi(B(x,  r)) \leq C_2^3 \mu(B(x,R/4)) = C_2^3 \mu(B(y_j,R/4))
\]	 
for all $j\in J$, one gets
	\[
	\# J C_2^{-3}\mi(B(x,  r)) \leq \sum_{j \in J} \mi(B(y_j,R/4))\leq\mi(B(x, r)),
	\]
	so that
\begin{equation}\label{cardJ}
	\# J\leq C_2^3.
\end{equation}
	Now, observe that
	\[
	\mi_w(B(x,r))\leq \sum_j \mi_w(B(y_j,R))\leq C_3^2 \sum_j \mi_w(B(y_j,R/4))\leq C_3^2 \mi_w(B(x,R)),
	\]
	and that, by~\eqref{cond2xr},
	\[
	\mi_w(B(x,R))\leq C_1^{-1}\mi_w(B(y_j,R/4))\left( \frac{\mi(B(x,R))}{\mi(B(y_j,R/4))}  \right)^p\leq C_1^{-1} C_2^{2p} \mi_w(B(y_j,R/4)),
	\]
	so that
\begin{equation}\label{C4}
	\mi_w(B(x,r))\leq C_4 \mi_w(B(y_j,R))
\end{equation}
with $C_4\coloneqq C_3^2 C_1^{-1} C_2^{2p}$, for every $j\in J$.

Finally, observe that we may find a partition $(F_j)_{j\in J}$ of $F$ into Borel sets such that $F_j\subseteq B(y_j,R)$ for every $j\in J$. Hence, by~\eqref{C4} and~\eqref{cond2xr}
\begin{align*}
		\frac{\mi_w(F)}{\mi_w(B(x,r))} &=\sum_{j} \frac{\mi_w(F_j)}{\mi_w(B(x,r))} \\
		& \geq  C_4^{-1}\sum_{j} \frac{\mi_w(F_j)}{\mi_w(B(y_j,R))} \geq C_1C_4^{-1} \sum_j \left(\frac{\mi(F_j)}{\mi(B(y_j,R))}\right)^{p},
\end{align*}	
while by~\eqref{cardJ} and~\eqref{C2},
			\[
	\begin{split}
\sum_j \left(\frac{\mi(F_j)}{\mi(B(y_j,R))}\right)^{p}
		&\geq  \# J^{1-p} \bigg(\sum_j\frac{\mi(F_j)}{\mi(B(y_j,R))}\bigg)^{p}\\
		&\geq C_2^{3(1-p)} C_2 ^{-2 p} \bigg(\sum_j\frac{\mi(F_j)}{\mi(B(y_j,R/4))}\bigg)^{p}\\
		&\geq  C_2^{3-5p} \bigg(\sum_j\frac{\mi(F_j)}{\mi(B(x,r))}\bigg)^{p}\\
		&\geq C_2 ^{3-5p} \bigg(\frac{\mi(F)}{\mi(B(x,r))}\bigg)^{p}.
	\end{split}
	\]
This shows that $w\in \widehat A_{p, (5/4)R}$. By iteration, one gets $w\in \widehat A_{p, (5/4)^k R}$ for all $k\in \N$. Since $\widehat  A_{p,R} \subseteq \widehat  A_{p,R'} $ if $R>R'$, the claim~\eqref{claimwidehatA} follows.
\end{proof}
As a corollary, we also see that if the condition  defining $A_{p,\loc}$ holds for some $R$, then it holds for all $R$'s.
\begin{corollary}\label{corAp}
Suppose $p\in (1,\infty)$. If $w\in A_{p,\loc}$, then for every $R>0$ there is $C>0$ such that
	\[
	\bigg(\dashint_{B} w\,\dd \mi \bigg)\bigg( \dashint_{B} w^{-p'/p}  \bigg)^{p/p'}\leq C
	\]
	for every $B\in \Bc_R$. Moreover, $A_{p,\loc} \subseteq \widetilde A_{\infty,\loc}$ for every $p\in [1,\infty)$.
\end{corollary}

\begin{proof}
	Observe first that by the implication (3)$\implies$(2) in Proposition~\ref{prop:ST}, the measures  $\mi$ and $\mi_w$ are equivalent, so that $w(x)>0$ for almost every $x\in G$. 

	Then,  $w^{-p'/p}\in A_{p',\loc}$.  Now,  as the proof of the implication  (3)$\implies$(4) in Proposition~\ref{prop:ST} shows, for every $R>0$ we may find $\delta>0$ such that
	\[
	\mi\bigg(\bigg\{ x\in B\colon w(x)\geq \delta \dashint_{B} w\,\dd \mi   \bigg\}\bigg)\geq \frac 2 3 \mi(B)
	\]
	and
	\[
	\mi\bigg(\bigg\{x\in B\colon w(x)^{-p'/p}\geq \delta \dashint_{B} w^{-p'/p}\,\dd \mi  \bigg\} \bigg)\geq \frac 2 3 \mi(B)
	\]
	for every $B\in \Bc_R$. Hence, 
	\[
	\mu\bigg( \bigg\{x\in B\colon w(x)\geq \delta \dashint_{B} w\,\dd \mi , \:\: w(x)^{-p'/p}\geq \delta \dashint_{B} w^{-p'/p}\,\dd \mi  \bigg\}\bigg) \geq \frac 1 3 \mi(B),
	\]
and in particular there exists $x_0$ belonging to the set in the left hand side. Then	
	\[
	\bigg(\dashint_{B} w\,\dd \mi \bigg)\left( \dashint_{B} w^{-p'/p}\,\dd \mi \right) ^{-p/p'}\leq \frac{w(x_0)}{\delta} \frac{w(x_0)^{-1}}{\delta^{p/p'}}=\frac{1}{\delta^{1+p/p'}}
	\]
for every $B\in \Bc_R$,  whence the first statement. To conclude, notice that by combining the fact that $A_{p,\loc} \subseteq A_{\infty, \loc}$ with the implication (3)$\implies$(2) in Proposition~\ref{prop:ST}, with the first part of the statement and with the local doubling property of $\mu$, see~\eqref{localdoubling}, one gets $A_{p,\loc}\subseteq \widetilde A_{\infty,\loc}$ when $p>1$. The case $p=1$ follows as well since $A_{1,\loc}\subseteq A_{p,\loc}$ for all $p>1$.
\end{proof}

\begin{proposition}\label{prop:6}
If the function $r \mapsto \mi(B(e,r))$ is continuous on $(0,R]$ for some $R>0$, then $A_{\infty,\loc}=\widetilde A_{\infty,\loc}$.
\end{proposition}

\begin{proof}
	Take $w\in A_{\infty,\loc}$. Then, by Proposition~\ref{prop:ST} there are $R'>0$ and $\eps,\delta\in (0,1)$ such that
\begin{equation}\label{oldclaim}
	\mi(F)\leq \eps \mi(B) \implies \mi_w(F)\leq \delta \mi_w(B) 
\end{equation}
	for every ball $B \in \mathcal{B}_{R'}$, and every Borel subset $F$ of $B$. We may assume that $R'<R/2$.
	
	Since $\mi$ is locally doubling, there is a constant $C>1$ such that
	\[
	\mi(B(e,2r))\leq C \mi(B(e, r))
	\]
	for every $r\in (0,R]$. Take  the smallest $\ell\in \N^*$ such that $(1-\eps)^\ell\leq C^{-1}$,  and observe that, by the continuity of the map $r\mapsto \mi(B(e,r))$ on  $(0,R]$, for every $r\in (0,R/2]$ we may find $r=r_0<\cdots<r_\ell=2 r$ such that
	\[
	\mi(B(x,r_{j-1}))\geq (1-\eps) \mi(B(x,r_{j}))
	\]
	for every $j=1,\dots, \ell$  and $x\in G$. In addition, $\ell\leq 1-\frac{\log C}{ \log(1-\eps)} \eqqcolon \ell^*$.
	 Therefore, by~\eqref{oldclaim},
	\begin{align*}
		\mu_w (B(x,r))=\mu_w (B(x,r_0))	&	\geq (1-\delta) \mu_w (B(x,r_1))\\
		& \geq \cdots \geq (1-\delta)^\ell\mu_w(B(x,r_\ell))\\
		&=(1-\delta)^\ell\mu_w(B(x,2r))\\
		& \geq (1-\delta)^{\ell^*}\mi_w(B(x,2r)),
	\end{align*}
whence $w\in \widetilde A_{\infty,\loc}$. 
\end{proof}

The proof of the following result is inspired by~\cite[Theorem 6]{Dallara}, which deals with cubes instead of balls, and with matrix-valued potential. The insterested reader can actually adapt our proof (and Corollary~\ref{corprop:3}, which is used therein) to the matrix-valued case, but we limit ourselves to the scalar case under consideration. See also~\cite{Auscher-BenAli}.

\begin{theorem}\label{teo:Muck}
Assume that $V\in \widetilde A_{\infty,\loc}$. Then, $\Hs_V$ has purely discrete spectrum if and only if
\begin{equation}\label{eqintinf}
\lim_{x\to \infty}\int_{B(x,R)} V\, \dd \mu =\infty
\end{equation}
for some (equivalently for all) $R>0$.
\end{theorem}
Observe that $V\in\widetilde A_{\infty,\loc} $ if $V\in A_{p,\loc}$ for some $p\in [1,\infty)$ by Corollary~\ref{corAp}, or if $V\in A_{\infty,\loc}$ and the map $r \mapsto \mi(B(x,r))$ is continuous on $(0,R]$ for some $R>0$ by Proposition~\ref{prop:6}.

\begin{proof}
If $\Hs_V$ has purely discrete spectrum, the statement follows by Corollary~\ref{corprop:3}. It is then enough to show that, if~\eqref{eqintinf} holds and $V\in \widetilde A_{\infty, \loc}$, then $\Hs_V$ has purely discrete spectrum. By Proposition~\ref{prop1}, it is enough to show that 
\begin{equation}\label{suffMuck}
\lim_{r\to \infty} \sup_{\substack{f\in C^\infty_c\\ Q(f)\leq 1}}\int_{G\setminus B(e,r)} |f |^2 =0.
\end{equation}
For a ball $B$ of radius $r>0$, let us define
\[
M(B) = r^2 \dashint_{B} V\, \dd \mu.
\]
Observe that, by Proposition~\ref{prop:5} and~\eqref{pallepiccole}, there is a constant $C>0$ such that
\[
M(2B)\leq C M(B), \qquad \forall B\in \Bc_{R/2}.
\]
Define $\Bc'_{R'}$, for $R'>0$, as the set of balls of radius $R$ which meet $G\setminus B(e,R')$. By~\eqref{eqintinf}, for all $k\in\mathbb{N}^*$ there is $R_k \geq k$ such that $M(B) \geq C^k$ for all $B\in \Bc'_{R_k}$. Therefore,
\[
M(2^{-k}B) \geq C^{-k} M(B) \geq 1
\]
for all $B\in \Bc'_{R_k}$. By Proposition~\ref{prop:5}, there are two constants $c,\delta>0$ such that, for all $B\in \Bc_R$,
\begin{equation}\label{mutildemu}
\mu(E_\delta(B)) \geq c \mu(B).
\end{equation}
Hence, for all $B\in \Bc'_{R_k}$ and $y\in  E_\delta(2^{-k} B)$,
\begin{align}\label{V2kdelta}
V(y) \geq \delta \dashint_{2^{-k} B} V \, \dd \mu = \frac{\delta 4^k}{R^2} M({2^{-k} B})\geq \frac{\delta 4^k}{R^2}.
\end{align}
Let now $f \in C_c^\infty$ be such that $Q(f) \leq 1$, and  $B\in \Bc'_{R_k}$. Then, by~\eqref{mutildemu} and~\eqref{V2kdelta}, 
\begin{align}\label{eq33}
\dashint_{{2^{-k} B}}|f|^2 \, \dd \mu 
& \leq 2 \dashint_{ {E_\delta(2^{-k} B)} \times {(2^{-k} B)}}|f(y) - f(y')|^2 \, \dd \mu(y)\, \dd\mu(y') \nonumber \\
& \qquad \qquad + 2 \dashint_{{E_\delta(2^{-k} B)}}|f(y)|^2 \, \dd \mu(y)\nonumber \\
& \leq \frac{2}{c} \dashint_{{(2^{-k} B)}\times {(2^{-k} B)}}|f(y) - f(y')|^2 \, \dd \mu(y)\, \dd\mu(y')\nonumber \\
& \qquad \qquad + \frac{2 R^2}{c\delta 4^k}\dashint_{ {2^{-k} B}} V |f|^2 \, \dd \mu.
\end{align}
Denote by  $C_{R}$ the constant in the Poincaré inequality~\eqref{Poincare}, and observe that
\begin{equation}\label{prelstep}
\dashint_{{2^{-k} B}}|f|^2 \, \dd \mu  \leq C_R \frac{2 R^2}{c4^k} \dashint_{{2^{-k} B}} |\nabla_\Hc f|^2 \, \dd \mu +  \frac{2 R^2}{c\delta 4^k}\dashint_{{2^{-k} B}} V |f|^2 \, \dd \mu.
\end{equation}
Let now $F_k$ be a  subset of $G\setminus B(e,R_k)$ which is maximal for the property that $d(y,y')\geq 2^{-k}R$ for every $y,y'\in F_k$, $y\neq y'$. Then, by maximality,
\[
G\setminus B(e,R_k)\subseteq \bigcup_{x\in F_k} B(x, 2^{-k}R).
\]
In addition, if $x\in F_k$ and 
\[
F_{k,x}\coloneqq \{y\in F_k\colon  d(x,y)<2^{-k+1}R\},
\]
then the balls $B(y,2^{-k-1}R)$, for $y\in F_{k,x}$, are pairwise disjoint and contained in $B(x, 2^{-k+2}R)$. Therefore, by the left-invariance of $\mi$,
\[
\# F_{k,x}\leq \frac{\mi(B(e,2^{-k+2}R))}{\mi(B(e,2^{-k-1}R))}\leq C'^3,
\]
where $C'$ is the doubling constant of $\mi$ for balls of radii at most $R$. Then, 
\[
\begin{split}
	\int_{G\setminus B(e,R_k)}\abs{f}^2\,\dd \mi&\leq \sum_{x\in F_k} \int_{B(x,2^{-k}R)} \abs{f}^2\,\dd \mi\\
		&\leq \max(C_R,\delta^{-1}) \frac{2R^2}{c 4^k}\sum_{x\in F_k}\int_{B(x,2^{-k}R)} (\abs{\nabla_\Hc f}^2+ V \abs{f}^2)\,\dd \mi\\
		&\leq \max(C_R,\delta^{-1})  C'^3\frac{2R^2}{c 4^k} Q(f) \leq \max(C_R,\delta^{-1}) C'^3\frac{2R^2}{c 4^k}.
\end{split}
\]
 Since the function 
\[
r\mapsto \sup_{\substack{f\in C^\infty_c\\ Q(f)\leq 1}}\int_{G\setminus B(e,r)} |f |^2 
\]
is decreasing, and since $R_k\to +\infty$ for $k\to \infty$,~\eqref{suffMuck} follows.
\end{proof}

\section{Weighted Sub-Laplacians}\label{Sec:WSL}

In this section, we denote by $w$ a positive function in $L^1_\loc$ such that $w^{-1}\in L^1_\loc$. As in the preceding section, we denote by $\mu_w$ the  measure with density $w$ with respect to $\mu$.  We consider the positive Hermitian form
\[
Q_w\colon (f,g)\mapsto \int_G  \nabla_\Hc f \cdot \nabla_\Hc \overline{g} \; \dd \mu_w,
\]
with domain
\[
\Dom(Q_w)= \{f\in L^2(\mu_w)\colon \nabla_\Hc f \in L^2_\Hc(\mu_w)\}.
\]
We emphasize that if $f\in L^2(\mu_w)$, then $ fw^{1/2}\in L^2$, so that $f=(f w^{1/2})w^{-1/2}\in L^1_\loc $. Hence, $\nabla_\Hc f$ is well defined in the distributional sense.

Observe that $Q_w$ is closed, continuous, and positive, and that $C^\infty_c\subseteq \Dom(Q_w)$, so that $Q_w$ is also densely defined. Then, it defines a positive self-adjoint operator $\Ls_w $ on $L^2(\mu_w)$ such that
\[
Q_w(f,g) = \langle \Ls_w f, g \rangle_{L^2(\mu_w)}
\]
for all $f\in \Dom(\Ls_w)$ and $g\in \Dom(Q_w)$.

\begin{proposition}
The space $C^\infty_c$ is dense in $\Dom(Q_w)$ with respect to the graph norm. If in addition $\nabla_\Hc w\in w^{1/2} L^2_{\Hc,\loc}$, then $C^\infty_c\subseteq \Dom(\mathcal{L}_w)$ and for $f\in C^\infty_c$
	\[
	\Ls_w f=\Ls f - \frac{\nabla_\Hc w}{w} \cdot \nabla_\Hc f.
	\]
\end{proposition}

\begin{proof}
Mutatis mutandis, the density of $C_c^\infty$ can be shown as in Lemma~\ref{lem:5}.

Assume now that $\nabla_\Hc w\in w^{1/2} L^2_{\Hc,\loc} \subseteq L^1_{\Hc,\loc}$, and observe that also
	\[
	\frac{\nabla_\Hc w}{w}\in w^{-1/2} L^2_{\Hc,\loc}\subseteq L^1_{\Hc,\loc}.
	\]
	Then, for every $f,g\in C^\infty_c$,
	\[
	\begin{split}
	Q_w(f,g)&=\int_G  \nabla_\Hc f \cdot \nabla_\Hc \overline{g}\,  w\, \dd \mu\\
		&=\int_G \Big((\Ls f )\overline g w -  (\nabla_\Hc f \cdot \nabla_\Hc w) \, \overline g\Big) \,\dd \mu\\
		&=\left\langle \Ls f- \frac{\nabla_\Hc w}{w}\cdot \nabla_\Hc f \, \bigg\vert\; g \right\rangle_{L^2(\mu_w)}.
	\end{split}
	\]
	Since $\Ls f- \frac{\nabla_\Hc w}{w} \cdot\nabla_\Hc f \in L^2(\mu_w)$, and since $C^\infty_c$ is dense in $\Dom(Q_w)$, this implies that $\Dom(\Ls_w)$ contains $C^\infty_c$ and
	\[
	\Ls_w f=\Ls f - \frac{\nabla_\Hc w}{w} \cdot \nabla_\Hc f
	\]
	for every $f\in C^\infty_c$.
\end{proof}

Assume now that $\nabla_\Hc w\in  L^1_{\Hc,\loc}\cap  (w L^2_{\Hc,\loc})$, $\Ls w\in  L^1_\loc  \cap   (w L^1_{\loc})$, so that the associated potential
\begin{equation}\label{potentialVw}
	V_w = -\frac{\abs{\nabla_\Hc w}^2}{4 w^2}-\frac{\Ls w}{2 w}
\end{equation}
	is in $L^1_\loc $. We also assume that $V_w$ is bounded from below, say $V_w \geq -m+1$. Observe that due to the specific form of such potential, one cannot easily reduce to the case $V\geq 1$ as before. Then, the corresponding Hermitian form
	\[
	Q_{V_w}\colon (f,g)\mapsto \int_G( \nabla_\Hc \cdot  \nabla_\Hc \overline{g} +V_w f \overline g)\,\dd \mu,
	\]
	with domain
	\[
	\Dom(Q_{V_w}) = \{f\in L^2 \colon \nabla_\Hc f\in L^2_{\Hc}, \;  f\sqrt{V_w+m}\in L^2\},
	\]
defines a Schr\"odinger operator $\Hc_{V_w}$ on $L^2$. We endow $\Dom(Q_{V_w})$ with the norm 
\[
f\mapsto \sqrt{m\norm{f}_{L^2}^2+Q_{V_w}(f)}\, .
\]
In the next proposition, we show that under the under slightly stronger, but necessary, assumptions the weighted sub-Laplacian $\Ls_w$ and the Schr\"odinger operator $\Hc_{V_w}$ are unitarily equivalent.

\begin{proposition}
	Assume that 
\begin{equation}\label{assumptionsw}
	\nabla_\Hc w\in (\sqrt{w} L^2_{\Hc,\loc})\cap  (w L^2_{\Hc,\loc}), \qquad \Ls w\in  L^1_\loc  \cap   (w L^1_{\loc}),
\end{equation}
	and that the associated potential satisfies $V_w \geq -m+1$ for some $m> 0$. Then, the map $f\mapsto f w^{1/2}$ induces an isometry of $\Dom(Q_w)$, endowed with the norm 
		\[
		f \mapsto \sqrt{m\norm{f}^2_{L^2(w)}+Q_w(f)}\, ,
		\]
	 onto $\Dom(Q_{V_w})$. In particular, it intertwines $Q_w$ and $Q_{V_w}$.
\end{proposition}

\begin{proof}
By the assumptions in~\eqref{assumptionsw}, $\nabla_\Hc w \in L^1_{\Hc,\loc}$ and $V_w \in L^1_\loc$. Let us first notice that
\begin{equation}\label{gradientiw}
\nabla_\Hc (w^{1/2}) = \frac{\nabla_\Hc w}{2w^{1/2}} \qquad \text{and} \qquad \nabla_\Hc (w^{-1/2}) = -\frac{\nabla_\Hc w}{2w^{3/2}}
\end{equation}
in the distributional sense. These formulae hold for $w+\epsilon$ in place of $w$ by the smoothness of the functions $t\to \sqrt{t}$ and $t \to t^{-1/2}$ on $(\epsilon/2, +\infty)$. Then, by our assumptions, the fact that $(w+\eps)^{-1/2}\leq w^{-1/2}$ and by dominated convergence, we see that the right hand sides in~\eqref{gradientiw}, for $w+\epsilon$ in place of $w$, converge in $L^1_\loc$ to the right hand sides in~\eqref{gradientiw}. Analogously, since $(w+\eps)^{\pm1/2}$ converge to $w^{\pm1/2}$ in $L^1_\loc$ by dominated convergence, the left hand sides in~\eqref{gradientiw}, for $w+\epsilon$ in place of $w$, converge in the sense of distributions to the left hand sides in~\eqref{gradientiw}.

We can now show that $fw^{1/2} \in \Dom(Q_{V_w})$ for every $f\in C^\infty_c$. By our assumptions,  $fw^{1/2}\in L^2$ and
\[
\nabla_\Hc(fw^{1/2}) = (\nabla_\Hc f) w^{1/2} + f \frac{\nabla_\Hc w}{ 2 w^{1/2}} \in L^2_\Hc.
\]
Moreover, since 
\[
w V_w = - \frac{|\nabla_\Hc w|^2}{4w} - \frac{1}{2}\Ls w \in L^1_\loc,
\]
one gets $w(V_w+m) \in L^1_\loc$, and hence $f \sqrt{w} \sqrt{V_w+m} \in L^2$, that is, $f \sqrt{V_w+m}\in L^2(\mi_w)$. 

A simple   integration by parts now shows that
	\[
	Q_{V_w}(f w^{1/2}, g w^{1/2})=Q_w(f,g)
	\]
	for all $f,g\in C_c^\infty$.
	
	Since $C_c^\infty$ is dense in $\Dom(Q_w)$, the map $f\mapsto f\sqrt{w}$ extends to an isometry
	\[
	I \colon \Dom(Q_w) \to \Dom(Q_{V_w}).
	\]
Since $\Dom(Q_w)$ and $\Dom(Q_{V_w})$ embed continuously in $L^2(\mi_w)$ and $L^2$ respectively, $I$ is the restriction on $\Dom(Q_w)$ of the surjective isometry  $J\colon L^2(\mi_w) \to L^2$ given by $J(f) = f\sqrt{w}$. To show that $I$ is surjective, it is enough to show that $J^{-1} (C_c^\infty) \subseteq \Dom(Q_w)$, so that $ C_c^\infty \subseteq I(\Dom(Q_w))$. Indeed, this implies that the image of $I$, which is closed since $I$ is an isometry, is dense in $\Dom(Q_{V_w})$; whence $I(\Dom(Q_w)) = \Dom(Q_{V_w})$.

Since $J$ is an isometry, in order to show that $fw^{-1/2} \in \Dom(Q_w)$ when $f\in C_c^\infty$, it will suffice to prove that $\nabla_\Hc(fw^{-1/2}) \in L^2_\Hc(\mi_w)$. This is a consequence of the equality
\[
\nabla_\Hc(f w^{-1/2})  = (\nabla_\Hc f) w^{-1/2} - \frac{1}{2} f (\nabla_\Hc w) w^{-3/2},
\]
and the assumptions~\eqref{assumptionsw} on $w$.
\end{proof}

\end{document}